\documentclass[twocolumn]{autart}

\usepackage{natbib}            
\usepackage{color}
\usepackage[usenames,dvipsnames]{xcolor}
\usepackage{subfigure}
\usepackage{enumerate}

\usepackage{multicol}
\usepackage{graphicx,amssymb,amstext,amsmath}
\usepackage{float}
\usepackage{empheq}

\makeatletter
\renewcommand*{\@opargbegintheorem}[3]{\trivlist
  \item[\hskip \labelsep{\bfseries #1\ #2}] \textbf{(#3)}\ \itshape}
\makeatother

\usepackage{algorithm,algorithmicx,algpseudocode}
\algnewcommand{\algorithmicgoto}{\textbf{go to}}%
\algnewcommand{\Goto}[1]{\algorithmicgoto~\ref{#1}}%
\algnewcommand{\LineComment}[1]{\Statex \(\triangleright\) #1}

\usepackage{multirow}
\usepackage{stfloats}
\usepackage{flushend}
\usepackage{cuted}

\makeatletter
\newcommand\footnoteref[1]{\protected@xdef\@thefnmark{\ref{#1}}\@footnotemark}
\makeatother

\renewcommand{\qed}{\hfill\rule{1.2ex}{1.2ex}}
\newenvironment{proof}[1][Proof]{\begin{trivlist}
\item[\hskip \labelsep {\textit{ #1}.}]}{\qed \end{trivlist}}

\let\oldFootnote\footnote
\newcommand\nextToken\relax

\renewcommand\footnote[1]{%
    \oldFootnote{#1}\futurelet\nextToken\isFootnote}

\newcommand\isFootnote{%
    \ifx\footnote\nextToken\textsuperscript{,}\fi}

\begin{document}

\begin{frontmatter}
% paper title
\title{A Unified Filter for Simultaneous Input and State Estimation of Linear Discrete-time Stochastic Systems}

\author[First]{Sze Zheng Yong}
\author[Second]{Minghui Zhu}
\author[First]{Emilio Frazzoli}

\address[First]{Laboratory for Information and Decision Systems, Massachusetts Institute of Technology, Cambridge, MA 02139, USA (e-mail: szyong@mit.edu, frazzoli@mit.edu).}
\address[Second]{Department of Electrical Engineering, Pennsylvania State University, 201 Old Main, University Park, PA 16802, USA (e-mail: muz16@psu.edu).}

\begin{abstract}
In this paper, we present a unified optimal and exponentially stable filter for linear discrete-time stochastic systems that simultaneously estimates the states and unknown inputs in an unbiased minimum-variance sense, without making any assumptions on the direct feedthrough matrix. We also derive input and state observability/detectability conditions, and analyze their connection to the convergence and stability of the estimator. We discuss two variations of the filter and their optimality and stability properties, and show  that filters in the literature, including the Kalman filter, are special cases of the filter derived in this paper. Finally, illustrative examples are given to demonstrate the performance of the unified unbiased minimum-variance filter.
 \end{abstract}

\end{frontmatter}

\section{Introduction}
\vspace{-0.2cm}
The term filter or estimator is commonly used to refer to systems that 
extract information about a quantity of interest from measured data 
corrupted by noise. Kalman filtering provides the tool needed for 
obtaining that reliable estimate when the system is linear and when 
the disturbance inputs or the unknown parameters are well modeled by 
a zero-mean, Gaussian white noise. However, in many instances, the exogenous input cannot be modeled as a Gaussian stochastic process rendering the estimates unreliable.

For example, consider the problem of estimating the state and 
inferring the intent of another vehicle at an intersection, 
for instance, for ensuring the safety of autonomous or semi-autonomous 
vehicles~\cite{Verma.2011}. In this case, the input of the other 
vehicle is inaccessible/unmeasurable, and is not well modeled by 
a zero-mean Gaussian white noise process. Thus, the standard Kalman 
filter does not yield an optimal estimate. Nonetheless, we want to 
be able to estimate the states and inputs of the other vehicle based on noisy
measurements for purposes of collision avoidance, route planning, etc.

Similar problems can be found across a wide range of disciplines, from the real-time estimation of mean areal precipitation during a storm \cite{Kitanidis.1987} to fault detection and diagnosis \cite{Patton.1989} to input estimation in physiological systems \cite{DeNicolao.1997}. Thus, this filtering problem in the presence of unknown inputs has steadily made it to the forefront in the recent decades.

\emph{Literature review.} Much of the research focus has been on state estimation of systems with unknown inputs without actually estimating the inputs. An optimal filter that estimates a minimum-variance unbiased (MVU) state estimate for a system with unknown inputs is first developed for linear systems without direct feedthrough in \cite{Kitanidis.1987}. This design was extended to a more general parameterized solution by \cite{Darouach.1997}, and eventually to state estimation of systems with direct feedthrough in \cite{Hou.1998,Darouach.2003,Cheng.2009}. %, the last of which is shown to be globally optimal over the class of all linear state estimators. 
Similarly, while $H_\infty$ filters (e.g., \cite{Simon.2006,Wang.et.al.2008,Wang.et.al.2013,Li.Gao.2013}) can deal with non-Gaussian disturbance inputs in minimizing the worst-case state estimation error, the unknown input is not estimated. However, the problem of estimating the unknown input itself is often as important as state information, and should also be considered.

Palanthandalam-Madapusi and Bernstein \cite{Palanthandalam.2007} proposed an approach to reconstruct the unknown inputs, in a process that is decoupled from state estimation with an emphasis on unbiasedness, but neglecting the optimality of the estimate. On the other hand,  Hsieh \cite{Hsieh.2000} and Gillijns and De Moor \cite{Gillijns.2007} developed simultaneous input and state filters that are optimal in the minimum-variance unbiased sense, for systems without direct feedthrough. Extensions to systems with a full rank direct feedthrough matrix were proposed by Gillijns and De Moor \cite{Gillijns.2007b}, Fang et al. \cite{Fang.2011} and Yong et al. \cite{Yong.Zhu.Frazzoli.2013}. In an attempt to deal with systems with a rank deficient direct feedthrough matrix, Hsieh \cite{Hsieh.2009} allowed the input estimate to be biased. Thus, the problem of finding a simultaneous state and input filter for systems with rank deficient direct feedthrough matrix that is both unbiased and has minimum variance remains open. Moreover, a unified MVU filter that works for all cases remains elusive.

Another set of relevant literature pertains to the stability of the state and input filters, since optimality does not imply stability and vice versa. However, to the best of our knowledge, the literature on this subject is limited to linear time-invariant systems \cite{Cheng.2009,Fang.2011,Fang.2011b}. Yet another related literature is on state and input observability and detectability conditions, also known as strong or perfect observability and detectability, as this will be shown to be related to the stability of the filter dynamics for both linear time-varying and time-invariant systems with unknown inputs. Some conditions for state and input observability were derived in \cite{Palanthandalam.2007,Hautus.1983,suda.Mutsuyoshi.1978}. 

\emph{Contributions.}
We introduce a unified filter for simultaneously estimating both state and unknown input such that the estimates are unbiased and have minimum variance with no restrictions on the direct feedthrough matrix of the linear discrete-time stochastic system. Within this framework, we propose two variants of the MVU state and input estimator, which are generalizations of the estimators in the literature, specifically of \cite{Gillijns.2007,Gillijns.2007b,Yong.Zhu.Frazzoli.2013}, and the Kalman filter. Furthermore, we derive sufficient conditions for the filter stability of linear time-varying systems with unknown inputs, an important problem that has been previously unexplored; while for linear time-invariant systems, necessary and sufficient conditions for convergence of the filter gains to a steady-state solution are provided. The key insight we gained is that the exponential stability of the filter is directly related to the strong detectability of the time-varying system, without which unbiased state and input estimates cannot be obtained even in the absence of stochastic noise. %, both of which we show to be related to the strong detectability of the system.
We shall also show that one of the filter variants we propose is globally optimal (i.e., optimal over the class of all linear state and input estimators as in \cite{Kerwin.Prince.2000}).

In connection to the existing literature, this paper presents a combination of several ideas from \cite{Cheng.2009,Gillijns.2007,Gillijns.2007b} and our recent work \cite{Yong.Zhu.Frazzoli.2013} into a unified filter in a manner that provably preserves and extends the nice properties of these filters. However, there are a number of distinctions between our filter and the above referenced filters. In particular, we show that the state-only filter in \cite{Cheng.2009} implicitly estimates the unknown inputs in a suboptimal manner and so does the approach for input estimation in \cite{Gillijns.2007b} (employed in one of the two variants of our filter). In contrast, our optimal filter variant uses the approaches of our previous work in \cite{Yong.Zhu.Frazzoli.2013} and of generalized least squares estimation, which lead to the desired optimality of the input estimates. In addition, we gave sufficient conditions for filter stability for linear time-varying systems, which clearly cannot be carried over from the existing literature (including \cite{Cheng.2009,Gillijns.2007,Gillijns.2007b}) for linear time-invariant systems.
 
%Finally, we present examples to illustrate both variants of the unified MVU filter.

\emph{Notation.} We first summarize the notation used throughout the paper. $\mathbb{R}^n$ denotes the $n$-dimensional Euclidean space, $\mathbb{C}$ the field of complex numbers and $\mathbb{N}$ nonnegative integers. For a vector of random variables, $v \in \mathbb{R}^n$, the expectation is denoted by $\mathbb{E}[v]$. Given a matrix $M \in \mathbb{R}^{p \times q}$, its transpose, inverse, Moore-Penrose pseudoinverse, range, trace and rank are given by $M^\top$, $M^{-1}$, $M^\dagger$, ${\rm Ra}(M)$, ${\rm tr}(M)$ and ${\rm rk}(M)$. For a symmetric matrix $S$, $S \succ 0$ and $S \succeq 0$ indicates that $S$ is positive definite and positive semidefinite, respectively.

\section{Problem Statement} \label{sec:Problem}
\vspace{-0.2cm}
\begin{figure*}[b]
\hrulefill
\begin{align} \label{eq:matrices}
\nonumber \mathcal{Z}_r &:= \begin{bmatrix} \mathcal{Z}_{1,r} \\ \mathcal{Z}_{2,r} \end{bmatrix}, \ \mathcal{Z}_{q,r} := \begin{bmatrix} z_{q,0}^\top & z_{q,1}^\top & \hdots & z_{q,r}^\top \end{bmatrix}^\top \, \forall \, q=\{1,2\} , \   \mathcal{Z}_{1,r} \in \mathbb{R}^{\sum^{r}_{k=0}p_{H_k}}, \ \mathcal{Z}_{2,r} \in \mathbb{R}^{(r+1)l-\sum^{r}_{k=0}p_{H_k}},  \\
\nonumber  \mathcal{D}_{r}&:=\begin{bmatrix} \mathcal{D}_{1,r} \\ \mathcal{D}_{2,r} \end{bmatrix}, \ \mathcal{D}_{1,r} :=\begin{bmatrix} d_{1,0}^\top & d_{1,1}^\top & \hdots & d_{1,r}^\top \end{bmatrix}^\top \in \mathbb{R}^{{\sum^{r}_{k=0}p_{H_k}}}, \mathcal{D}_{2,r} :=\begin{bmatrix}  d_{2,0}^\top & d_{2,1}^\top & \hdots & d_{2,r-1}^\top \end{bmatrix}^\top \in \mathbb{R}^{rp-{\sum^{r-1}_{k=0}p_{H_k}}}, \\
\nonumber p_{H_k}&={\rm rk}(H_k) \; \forall \, 0 \leq k \leq r, \ \Phi_{(i,i)} := \hat{A}_i := A_i-G_{1,i}\Sigma_{i}^{-1}C_{1,i}, \;  \Phi_{(i,j> i)} := \hat{A}_{j} \hdots \hat{A}_i, \; \makebox{and } \; \forall \,q =\{1,2\}, s=\{1,2\},\\
\nonumber \mathcal{O}_{q,r} &:=
\begin{bmatrix} C_{q,0} \\ C_{q,1}\hat{A}_0 \\  C_{q,2}\Phi_{(0,1)} \\ \vdots \\  C_{q,r-1}\Phi_{(0,r-2)} \\  C_{q,r}\Phi_{(0,r-1)}  \end{bmatrix}, \,
\mathcal{I}_{(q,s),r} := \begin{bmatrix}
0 & 0 & \hdots & 0 & 0 \\
C_{q,1}G_{s,0}  & 0 & \hdots & 0 & 0 \\
C_{q,2}\hat{A}_1 G_{s,0} & C_{q,2}G_{s,1} & \hdots & 0 & 0\\
\vdots & \vdots & \ddots & \vdots  & \vdots \\
C_{q,r-1}\Phi_{(1,r-2)}G_{s,0} & C_{q,r-1}\Phi_{(2,r-2)}G_{s,1} & \hdots & C_{q,r-1}G_{s,r-2} & 0 \\
C_{q,r}\Phi_{(1,r-1)}G_{s,0}  & C_{q,r}\Phi_{(2,r-1)}G_{s,1} & \hdots & C_{q,r}\hat{A}_{r-1}G_{s,r-2} & C_{q,r}G_{s,r-1}
\end{bmatrix}   \tag{$\star$} 
\end{align}
\end{figure*}

Consider the linear time-varying discrete-time system
\begin{align} \label{eq:system}
\begin{array}{ll}
x_{k+1}&=A_k x_k+B_k u_k+G_k d_k +w_k\\
y_k&=C_k x_k +D_k u_k + H_k d_k + v_k \end{array}
\end{align}
where $x_k \in \mathbb{R}^n$ is the state vector at time $k$, $u_k \in \mathbb{R}^m$ is a known input vector, $d_k \in \mathbb{R}^p$ is an unknown input vector, and $y_k \in \mathbb{R}^l$ is the measurement vector. The process noise $w_k \in \mathbb{R}^n$ and the measurement noise $v_k \in \mathbb{R}^l$ are assumed to be mutually uncorrelated, zero-mean, white random signals with known covariance matrices, $Q_k=\mathbb{E} [w_k w_k^\top] \succeq 0$ and $R_k=\mathbb{E} [v_k v_k^\top] \succ 0$, respectively. Without loss of generality, we assume throughout the paper that $n \geq l \geq 1$, $l \geq p \geq 0$ and $m \geq 0$, and that the current time variable $r$ is strictly nonnegative. $x_0$ is also assumed to be independent of $v_k$ and $w_k$ for all $k$. % and the unbiased estimate $\hat{x}_0$ of the initial state $x_0$ is available with covariance matrix $\mathcal{P}_{0}^x$.

The matrices $A_k$, $B_k$, $C_k$, $D_k$, $G_k$ and $H_k$ are known and bounded. Note that no assumption is made on $H_k$ to be either the zero matrix (no direct feedthrough), or to have full column rank when there is direct feedthrough. Without loss of generality, we assume ${\rm max}_k ({\rm rk}[G_k^\top\; H_k^\top ] )=p$. 
(Otherwise, we can retain the linearly independent columns and the ``remaining" inputs still affect the system in the same way.)

The estimator design problem, addressed in this paper, can be stated as follows:\newline
\emph{Given a linear discrete-time stochastic system with unknown inputs \eqref{eq:system}, design a globally optimal and stable filter that simultaneously estimates system states and unknown inputs in an unbiased minimum-variance manner.} 

\section{Preliminary Material}
\subsection{System Transformation}

We first carry out a transformation of the system to decouple the output equation into two components, one with a full rank direct feedthrough matrix and the other without direct feedthrough. In this form, the filter can be designed leveraging existing approaches for both cases (e.g., \cite{Gillijns.2007,Yong.Zhu.Frazzoli.2013}).

Let $p_{H_k}:={\rm rk} (H_k)$. Using singular value decomposition, we rewrite the direct feedthrough matrix $H_k$  as
\begin{align}
H_k= \begin{bmatrix}U_{1,k}& U_{2,k} \end{bmatrix} \begin{bmatrix} \Sigma_k & 0 \\ 0 & 0 \end{bmatrix} \begin{bmatrix} V_{1,k}^{\, \top} \\ V_{2,k}^{\, \top} \end{bmatrix}
\end{align}
where $\Sigma_k \in \mathbb{R}^{p_{H_k} \times p_{H_k}}$ is a diagonal matrix of full rank, $U_{1,k} \in \mathbb{R}^{l \times p_{H_k}}$, $U_{2,k} \in \mathbb{R}^{l \times (l-p_{H_k})}$, $V_{1,k} \in \mathbb{R}^{p \times p_{H_k}}$, $V_{2,k} \in \mathbb{R}^{p \times (p-p_{H_k})}$, and $U_k:=\begin{bmatrix} U_{1,k} & U_{2,k} \end{bmatrix}$ and $V_k:=\begin{bmatrix} V_{1,k} & V_{2,k} \end{bmatrix}$ are unitary matrices. Note that in the case with no direct feedthrough, $\Sigma_k$, $U_{1,k}$ and $V_{1,k}$ are empty matrices\footnote{We adopt the convention that the inverse of an empty matrix is also an empty matrix and assume that operations with empty matrices are possible. These features are readily available in many simulation software products such as MATLAB, LabVIEW and GNU Octave. Otherwise, a conditional statement can be included to bypass this case.}, and $U_{2,k}$ and $V_{2,k}$ are arbitrary unitary matrices.

Then, as suggested in \cite{Cheng.2009}, we define two orthogonal components of the unknown input given by
\begin{align}
d_{1,k}=V_{1,k}^\top d_k, \quad
d_{2,k}=V_{2,k}^\top d_k.
\end{align}
Since $V_k$ is unitary, $d_k =V_{1,k} d_{1,k}+V_{2,k} d_{2,k}$ and the system \eqref{eq:system} can be rewritten as
\begin{align}
\nonumber x_{k+1}&=A_k x_k+B_k u_k+G_k V_{1,k} d_{1,k} +G_k V_{2,k} d_{2,k} +w_k\\
& = A_k x_k+B_k u_k+G_{1,k} d_{1,k} +G_{2,k} d_{2,k} +w_k \quad \; \label{eq:sysX}\\
\nonumber y_k&=C_k x_k +D_k u_k + H_k V_{1,k} d_{1,k} + H_k V_{2,k} d_{2,k}  + v_k\\
&=C_k x_k +D_k u_k + H_{1,k} d_{1,k} + v_k, \label{eq:y}
\end{align}
where $G_{1,k} :=G_k V_{1,k}$, $G_{2,k} :=G_k V_{2,k}$ and $H_{1,k} :=H_k V_{1,k}=U_{1,k} \Sigma_k$. Next, as aforesaid, we decouple the output $y_k$ 
using a nonsingular transformation
\begin{align} \label{eq:T_k}
\nonumber T_k &=\begin{bmatrix} T_{1,k} \\ T_{2,k} \end{bmatrix} \\
&= \begin{bmatrix} I_{p_{H_k}} & -U_{1,k}^\top R_k U_{2,k} (U_{2,k}^\top R_k U_{2,k})^{-1}\\ 0 & I_{(l-p_{H_k}) } \end{bmatrix} \begin{bmatrix} U_{1,k}^\top \\ U_{2,k}^\top \end{bmatrix}
\end{align}
to obtain $z_{1,k} \in \mathbb{R}^{p_{H_k}}$ and $z_{2,k} \in \mathbb{R}^{l-p_{H_k}}$ given by
\begin{align} \label{eq:sysY} \begin{array}{lll}
z_{1,k}&=T_{1,k} y_k &= C_{1,k} x_k + D_{1,k} u_k +\Sigma_k d_{1,k} + v_{1,k}\\
z_{2,k}&=T_{2,k} y_k &= C_{2,k} x_k + D_{2,k} u_k + v_{2,k}\end{array}
\end{align}
where $C_{1,k} :=T_{1,k} C_k$, $C_{2,k} := T_{2,k} C_k = U_{2,k}^\top C_k$, $D_{1,k} :=T_{1,k} D_k$, $D_{2,k} := T_{2,k} D_k = U_{2,k}^\top D_k$, $v_{1,k} :=T_{1,k} v_k$ and $v_{2,k} := T_{2,k} v_k = U_{2,k}^\top v_k$. This transform is also chosen such that the measurement noise terms for the decoupled outputs are uncorrelated. The covariances of $v_{1,k}$ and $v_{2,k}$ can then be found as follows:
\begin{align} \label{eq:R12}
\nonumber R_{1,k}&:=\mathbb{E}[v_{1,k} v_{1,k}^\top]=T_{1,k} R_k T_{1,k}^\top \succ 0\\
\nonumber R_{2,k}&:=\mathbb{E}[v_{2,k} v_{2,k}^\top]=T_{2,k} R_k T_{2,k}^\top = U_{2,k}^\top R_k U_{2,k}  \succ 0 \\
 R_{12,(k,i)}&:=\mathbb{E}[v_{1,k} v_{2,i}^\top] = T_{1,k} R_k T_{2,k}^\top = U_{1,k}^\top R_k U_{2,k}^\top \hspace{-0.2cm}\\
\nonumber & \hspace{-1.5cm}- U_{1,k}^\top R_k U_{2,k} (U_{2,k}^\top R_k U_{2,k})^{-1}U_{2,k}^\top R_k U_{2,k} =0, \; \forall k, i \in \mathbb{N}
 \end{align}
Since the initial state, process and measurement noise, are assumed to be uncorrelated, the covariances of $v_{1,k}$ and $v_{2,k}$ with the initial state and process noise are
\begin{align}
\nonumber \mathbb{E}[v_{1,k} w_{i}^\top] &= T_{1,k} \mathbb{E}[v_k w_i^\top]=0 \\
\nonumber \mathbb{E}[v_{2,k} w_{i}^\top] &= T_{2,k} \mathbb{E}[v_k w_i^\top]=0 \\
 \mathbb{E}[v_{1,k} v_{1,i}^\top] &= T_{1,k} \mathbb{E}[v_k v_i^\top] T_{1,i}^\top=0, \; \forall k \neq i\\
\nonumber \mathbb{E}[v_{2,k} v_{2,i}^\top] &= T_{2,k} \mathbb{E}[v_k v_i^\top] T_{2,i}^\top=0, \; \forall k \neq i\\
\nonumber \mathbb{E}[v_{1,k} x_0^\top] &= T_{1,k} \mathbb{E}[v_k x_0^\top] =0\\
\nonumber \mathbb{E}[v_{2,k} x_0^\top] &= T_{2,k} \mathbb{E}[v_k x_0^\top] =0.
\end{align}

\subsection{Input and State Observability and Detectability} \label{sec:obsv}

Similar to the analysis of the convergence of the Kalman filter, we will show in Section \ref{sec:analysis} that the convergence of the unified filter is directly related to the notion of input and state observability and detectability (with $w_k = v_k = 0$, and without loss of generality, we assume that $B_k=D_k=0$), also known as strong or perfect observability and detectability (e.g., see (\cite{Hautus.1983,suda.Mutsuyoshi.1978,Silverman.1976}), defined as follows:

\begin{defn}[Strong observability]
The linear system \eqref{eq:system} is \emph{strongly observable}, 
or equivalently \emph{state and input observable} or \emph{perfectly observable}, if the initial 
condition $x_0$ and the unknown input sequence up to time $r-1$, 
$\{d_i\} ^{r-1}_{i=0}$, and specifically 
$\mathcal{D}_{r} \in \mathbb{R}^{rp+p_{H_r}} $, 
can be uniquely determined from the measured output sequence $\{y_i \}^r_{i=0}$, or equivalently $\mathcal{Z}_r \in \mathbb{R}^{(r+1)l} $, for a large enough number of observations, i.e.,  $r \geq r_0$ for some $r_0 \in \mathbb{N}$, where $\mathcal{D}_{r}$ and $\mathcal{Z}_r$ are given in \eqref{eq:matrices}. 
\end{defn}

Next, we present the conditions for strong observability for the time-varying and time-invariant cases.%
\begin{thm}[Strong observability (time-varying)] \label{thm:obsv}
A linear time-varying discrete-time system is input and state observable
if and only if 
\begin{align}\textstyle {\rm rk} (\begin{bmatrix} \mathcal{O}_{2,r} & \mathcal{I}_{(2,2),r} \end{bmatrix})=n+rp-{\sum^{r-1}_{k=0}p_{H_k}} \label{cond} \end{align}

where $p_{H_k}$, as well as the \emph{observability} and \emph{invertibility} matrices, $\mathcal{O}_{2,r} \in \mathbb{R}^{((r+1)l-{\sum^{r}_{k=0}p_{H_k}}) \times n}$ and 
$ \mathcal{I}_{(2,2),r} \in \mathbb{R}^{((r+1)l-{\sum^{r}_{k=0}p_{H_k}}) \times (rp-{\sum^{r-1}_{k=0}p_{H_k}})}$, are given in \eqref{eq:matrices}.

Necessary conditions for \eqref{cond} to hold are
\begin{enumerate}[(I)]
\item
$ r \geq r_{0,r}$ and $l \geq p+1$, or
$l=p=n$ and $p_{H_r} =0$,
\label{cond:1}
\item \begin{enumerate}[(a)] \item ${\rm rk}(\mathcal{O}_{2,r})=n$,
\item ${\rm rk}(\mathcal{I}_{(2,2),r}^k)=p-p_{H_{k-1}}$, $\forall \, 0 \leq k \leq r$,
\end{enumerate}
\end{enumerate}
where $r_{0,r}:=\lceil \frac{n-l-p_{H_r}}{l-p} \rceil$, $\lceil a \rceil$ is the smallest integer not less than $a$ and $\mathcal{I}_{(2,2),r}^k$ is the $k$-th column of $\mathcal{I}_{(2,2),r}$.
\end{thm}
\begin{proof}
The system transformation given by \eqref{eq:T_k} transforms the output equations such that the $d_{1,k}$ component of can be determined from only the current output measurement and previous state and input estimates. Specifically,
from \eqref{eq:sysX} and \eqref{eq:sysY}, and ignoring the
known input and noise terms, we find $\mathcal{D}_{1,r}=-\breve{\Sigma} \begin{bmatrix}\mathcal{O}_{1,r} & \mathcal{I}_{(1,2),r} \end{bmatrix} \begin{bmatrix} x_0 \\ \mathcal{D}_{2,r} \end{bmatrix} +(I_{\sum^{r}_{k=0}p_{H_k}}-\mathcal{I}_{(1,1),r})\breve{\Sigma} \mathcal{Z}_{1,r}$ where $\breve{\Sigma}=\begin{bmatrix} \Sigma_0^{-1}& \hdots & \Sigma_r^{-1} \end{bmatrix}$. Substituting this in the output equation $z_{2,k}$ in \eqref{eq:sysX}, we observe that the initial state $x_0$ and unknown input $\mathcal{D}_{2,r}$ (and consequently $\mathcal{D}_{1,r}$ from the previous equation) can be obtained from
$\begin{bmatrix} \mathcal{O}_{2,r} & \mathcal{I}_{(2,2),r} \end{bmatrix}
\begin{bmatrix} x_0 \\ \mathcal{D}_{2,r} \end{bmatrix}
=\mathcal{Z}_{2,r}-\mathcal{I}_{(2,1),r}
\begin{bmatrix} G_{1,0}\Sigma_0^{-1} & \hdots & G_{1,r-1}\Sigma_{r-1}^{-1} \end{bmatrix}
\mathcal{Z}_{1,r-1}$.
Thus, the linear system has a unique solution if and only if \eqref{cond} holds:
\begin{enumerate} [(I)]
\item The linear system is not underdetermined, i.e., $(r+1)l -{\sum^{r}_{k=0}p_{H_k}}\geq n +rp -{\sum^{r-1}_{k=0}p_{H_k}} \Rightarrow (r+1)l \geq n +rp+p_{H_r}$. Thus, \eqref{cond:1} holds.
\item The matrix $\begin{bmatrix} \mathcal{O}_{2,r} & \mathcal{I}_{(2,2),r} \end{bmatrix}$ has full column rank. For this to hold, the following are necessary:
\begin{enumerate}[(a)]
\item $\mathcal{O}_{2,r}$ has full column rank.
\item $\mathcal{I}^k_{(2,2),r}$ has full column rank, $\forall \, 1 \leq k \leq r$. \vspace{-0.06cm}
\end{enumerate}\end{enumerate}
\vspace{-0.7cm}
\end{proof}

\begin{thm}[Strong observability(time-invariant)] \label{thm:obsvTI}
A linear time-invariant discrete-time system is input and state observable if and only if
\begin{align}
{\rm rk} (\begin{bmatrix} \mathcal{O}_{2,\tilde{n}} & \mathcal{I}_{(2,2),\tilde{n}} \end{bmatrix})=n+\tilde{n}(p- p_{H}) \label{eq:tinv}
\end{align}
for some $0 \leq \tilde{n} \leq n$ where $p_H={\rm rk}(H)$. Moreover, if $l \neq p$, then $r_0 \leq \tilde{n} \leq n$, where $r_0:=\lceil \frac{n-l-p_H}{l-p} \rceil$; otherwise, $l=p=n$ and $p_H=0$ must hold. %\yong{$\tilde{n}=\max\{n,r_0\}$,}
Necessary conditions for \eqref{eq:tinv} to hold when $\tilde{n}=n$ are
\begin{enumerate}[(I)]
\item
$ r \geq r_0$ and $l \geq p+1$, or
$l=p=n$ and $p_H =0$,
\item \begin{enumerate}[(a)] \item ${\rm rk}(\mathcal{O}_{2,n-1})=n$; thus, $(\hat{A},C)$ is observable,
\item ${\rm rk}(C_2G_2)=p-p_H$; thus, ${\rm rk}(C G) \geq p-p_H$,
\end{enumerate}
\end{enumerate}
where $\hat{A}=A-G_1 \Sigma^{-1} C_1$.
\end{thm}
\begin{proof} By applying the Cayley-Hamilton theorem,
we can show that the observable subspace spanned by $\mathcal{O}_{2,n-1}$
is $\hat{A}$-invariant (i.e., ${\rm Ra}( \mathcal{O}_{2,n-1}\hat{A}) \subset {\rm Ra}(\mathcal{O}_{2,n-1})$), which implies that  ${\rm rk} (\mathcal{O}_{2,r})={\rm rk} (\mathcal{O}_{2,n-1})$ for all $r \geq n$. Then, to prove the conditions given in the theorem, we will show that (i) if $\begin{bmatrix} \mathcal{O}_{2,n} & \mathcal{I}_{(2,2),n} \end{bmatrix}$ is rank deficient, then $\begin{bmatrix} \mathcal{O}_{2,r} & \mathcal{I}_{(2,2),r} \end{bmatrix}$ for all $r>n$ is also rank deficient, and (ii) if $\begin{bmatrix} \mathcal{O}_{2,n} & \mathcal{I}_{(2,2),n} \end{bmatrix}$ has full rank, then $\begin{bmatrix} \mathcal{O}_{2,r} & \mathcal{I}_{(2,2),r} \end{bmatrix}$ for all $r>n$ also has full rank.
\begin{enumerate}[(i)]
\item Suppose $\begin{bmatrix} \mathcal{O}_{2,n} & \mathcal{I}_{(2,2),n} \end{bmatrix}$ is rank deficient. This implies one of three cases. In the first, $\mathcal{O}_{2,n}$ is rank deficient. This then implies that $\mathcal{O}_{2,r}$ for all $r>n$ is also rank deficient since ${\rm rk} (\mathcal{O}_{2,r})={\rm rk} (\mathcal{O}_{2,n})$. In the second case, one of the matrices $\{\mathcal{I}^d_{(2,2),n}\}^n_{d=1}$ ($d$-th column matrix of $\mathcal{I}_{(2,2),n}$ each of dimensions $r(l-p_H) \times p-p_H$) is rank deficient, which implies that $\mathcal{I}_{(2,2),r}^{d+r-n}=\begin{bmatrix} 0^\top_{(l-p_H)(r-n) \times (p-p_H)} & \mathcal{I}_{(2,2),n}^{d \top} \end{bmatrix}^\top$ is rank deficient for all $r>n$. And in the third case, some columns of some matrix pair between $\mathcal{O}_{(2,2),n}$ and $\{\mathcal{I}^d_{(2,2),n}\}^n_{d=1}$ are linearly dependent, which by virtue of the lower triangular structure of $\begin{bmatrix} \mathcal{O}_{2,n} & \mathcal{I}_{(2,2),n} \end{bmatrix}$ is only possible if some columns of either $C_{2}$ or $C_{2} G_{2}$ are zero vectors. However, this implies that $C_{2} G_2$ and hence $\mathcal{I}^r_{(2,2),r}=\begin{bmatrix} 0^T_{r(l-p_H) \times (p-p_H)} & (C_2 G_2)^\top\end{bmatrix}^\top$ is rank deficient for all $r>n$. Therefore, in all cases, $\begin{bmatrix} \mathcal{O}_{2,r} & \mathcal{I}_{(2,2),r} \end{bmatrix}$ for all $r>n$ is rank deficient.
\item Suppose now that $\begin{bmatrix} \mathcal{O}_{2,n} & \mathcal{I}_{(2,2),n} \end{bmatrix}$ has full rank. This implies that $\mathcal{O}_{2,n}$ and $\{ \mathcal{I}^d_{(2,2),n}\}^{n}_{d=1}$ have full rank, which in turn implies that for all $r >n$, $\mathcal{O}_{2,r}$ is full rank since ${\rm rk} (\mathcal{O}_{2,r})={\rm rk} (\mathcal{O}_{2,n})$ and $C_2 G_2$ is also full rank, which can be inferred from $\mathcal{I}_{(2,2),n}^n$ being full rank. Hence, since the matrices $\{ \mathcal{I}^d_{(2,2),r}\}^r_{d=1}$ have the form $\begin{bmatrix} 0 & (C_2G_2)^\top & (*) \end{bmatrix}^\top$ with $0$ and $(*)$ of appropriate entries and dimensions, each of these matrices $\{ \mathcal{I}^d_{(2,2),r}\}^r_{d=1}$ have full rank. %=\left\{\begin{bmatrix} 0^\top_{(d+r-n) (l-p_H) \times (p-p_H)} & \mathcal{I}^{d \top}_{(2,2),n} \end{bmatrix}^\top\right\}_{d=1}^n$.  Next, with the rest of  $\mathcal{I}_{(2,2),r}$ given by $\{ \mathcal{I}^d_{(2,2),r}\}^{r-n}_{d=1}=\left\{\begin{bmatrix} 0^\top_{d (l-p_H) \times (p-p_H)} & (\mathcal{O}_{2,r-d+1} G_2)^\top  \end{bmatrix}^\top\right\}_{d=1}^{r-n}=\left\{\begin{bmatrix} \mathcal{I}^{d \, \top}_{(2,2),n} & (C_2 \hat{A}^{d+n-1} G_2)^\top & \hdots & (C_2 \hat{A}^{r+d-2} G_2)^\top  \end{bmatrix}^\top\right\}_{d=1}^{r-n}$ where $r-d+1>n$ and with the first $(d+n-1)(l-p_H)$ columns equal to $\mathcal{I}_{(2,2),n}^d$, by the Cayley-Hamilton theorem, these matrices are also full rank. 
Finally, since the assumption also implies that $C_2$ and $C_2 G_2$ cannot have zero columns and 
the matrix $\begin{bmatrix} \mathcal{O}_{2,r} & \mathcal{I}_{(2,2),r} \end{bmatrix}$ has a lower triangular structure, then this matrix must also have full rank. 
\end{enumerate}

%Thus, noticing that $\mathcal{O}_{2,n-1}$ is embedded
%in the first column of $\mathcal{I}_{(2,2),n}$,
%it follows that \eqref{eq:tinv} is sufficient for guaranteeing
%the linear independence of the columns of $\begin{bmatrix} \mathcal{O}_{2,r} & \mathcal{I}_{(2,2),r} \end{bmatrix}$ for all $r \geq n$. Similarly, ${\rm rk}(\mathcal{O}_{2,n-1})=n$ is sufficient for satisfying ${\rm rk}(\mathcal{O}_{2,r})=n$  for all $r \geq n$. 
Note that an alternative proof can be found in \cite{Silverman.1976}.
Furthermore, since $\mathcal{O}_{2,n-1}=U_2^\top \mathcal{\hat{O}}$, where $\mathcal{\hat{O}}=\begin{bmatrix} C^\top & (C\hat{A})^\top & \hdots & (C\hat{A}^{n-1})^\top \end{bmatrix}^\top$ and $C_2G_2=U_2^\top CGV_2$, then ${\rm rk}(\mathcal{O}_{2,n})={\rm rk}(\mathcal{O}_{2,n-1}) \leq {\rm min}({\rm rk}(U_2^\top),{\rm rk}(\mathcal{O}))$ and ${\rm rk}(C_2 G_2) \leq {\rm min}({\rm rk}(U_2^\top),{\rm rk}(CG),{\rm rk}(V_2))$. Thus, it follows that ${\rm rk}(\mathcal{\hat{O}})=n$ (i.e., $(\hat{A},C)$ is observable) and ${\rm rk}(C G) \geq p-p_H$ are necessary. 
\end{proof}

\begin{cor} \label{cor:1}
For the time-invariant case, the following statements are equivalent:
\begin{enumerate}[(i)]
\item  ${\rm rk} (\begin{bmatrix} \mathcal{O}_{2,n} & \mathcal{I}_{(2,2),n} \end{bmatrix})=n+n(p- p_{H})$, \label{cond:i}
\item ${\rm rk}\begin{bmatrix} zI-A & -G \\ C & H \end{bmatrix}=n+p$ for all $z \in \mathbb{C}$, \label{cond:ii}
\item ${\rm rk}\left(\begin{bmatrix} zI-\hat{A} & -G_2 \\ C_2 & 0 \end{bmatrix}\right)=n+p-p_H$ for all $z \in \mathbb{C}$. \label{cond:iii}
\end{enumerate}
Moreover, the observability of $(A,C)$ is a necessary condition.
\end{cor}
\begin{proof} The proof of the equivalence of \eqref{cond:i} and \eqref{cond:ii} is fairly involved, and the reader is referred to \cite{suda.Mutsuyoshi.1978,Silverman.1976}  for details.
To relate \eqref{cond:ii} and \eqref{cond:iii}, we use the following \vspace{-0.2cm} %identity
\begin{align*}
&n+p={\rm rk}\begin{bmatrix} zI-A & -G \\ C & H \end{bmatrix} ={\rm rk}\begin{bmatrix} zI-A & -G \\ C & U \begin{bmatrix} \Sigma & 0 \\ 0 & 0 \end{bmatrix} V^\top \end{bmatrix} \\ 
&= {\rm rk} \begin{bmatrix} I & 0 \\ 0 & T \end{bmatrix} \begin{bmatrix} zI-A & -G \\ C & U \begin{bmatrix} \Sigma & 0 \\ 0 & 0 \end{bmatrix} V^\top \end{bmatrix} \begin{bmatrix} I & 0 \\ 0 & V \end{bmatrix} \\% \end{align*} \begin{align*} %\\
&={\rm rk}\begin{bmatrix} zI-A & -GV \\ TC & T U\begin{bmatrix} \Sigma & 0 \\ 0 & 0 \end{bmatrix} \end{bmatrix} = {\rm rk}  \begin{bmatrix} zI-A & -G_1 & -G_2\\ C_1 & \Sigma & 0\\ C_2 & 0 & 0 \end{bmatrix} \\ % \end{align*} \begin{align*} % \\
&={\rm rk} \begin{bmatrix} I & G_1\Sigma^{-1} & 0 \\ 0 & I & 0\\ 0 & 0 & I\end{bmatrix} \begin{bmatrix} zI-A & -G_1 & -G_2 \\ C_1 & \Sigma &0 \\ C_2 & 0 & 0 \end{bmatrix} \\ %\end{align*} \begin{align*} %\\
&={\rm rk}\begin{bmatrix} zI-\hat{A} & 0 & -G_2 \\ C_1 & \Sigma & 0\\ C_2 & 0 & 0 \end{bmatrix}={\rm rk}\begin{bmatrix} zI-\hat{A} & -G_2 \\ C_2 & 0 \end{bmatrix}+p_H
\end{align*}
for all $z \in \mathbb{C}$, where the final equality holds because $\Sigma$ is square and has full rank $p_H$. 
The necessity of observability of the pair $(A,C)$ follows directly from \eqref{cond:ii}.
\end{proof}

\begin{rem}
Note that if ${\rm rk}(H_r)=p$, then $d_{2,r}$ is empty and $\mathcal{D}_r$ contains unknown inputs up to time $r$. 
\end{rem}

A weaker condition than the strong observability is given in the following definition and theorem.
\begin{defn}[Strong detectability] \label{def:det}
The linear system \eqref{eq:system} is \emph{strongly detectable} if
\begin{align*}
y_k=0 \ \forall \, k \geq 0 \quad \textrm{implies} \quad x_k \to 0 \ \textrm{as} \ k \to \infty
\end{align*}
for all input sequences and initial states. %In other words, the system is strongly detectable if and only if all of its strongly unobservable modes are stable.
%(i.e., invariant zeros of the system matrices in Corollary \ref{cor:1}-\eqref{cond:ii},\eqref{cond:iii})
\end{defn}

\begin{thm}[Strong detectability(time-invariant)] \label{thm:det}
A linear time-invariant discrete-time system is strongly detectable if and only if either of the following holds: 
\begin{enumerate}[(i)]
\item ${\rm rk}\begin{bmatrix} zI-A & -G \\ C & H \end{bmatrix}=n+p$, $\forall z \in \mathbb{C}, |z| \geq 1$,
\item ${\rm rk}\begin{bmatrix} zI-\hat{A} & -G_2 \\ C_2 & 0 \end{bmatrix}=n+p-p_H$, $\forall z \in \mathbb{C}, |z| \geq 1$.
\end{enumerate}
The above conditions are equivalent to the property that the system is minimum-phase (i.e., the invariant zeros of the system matrices in Corollary \ref{cor:1}-\eqref{cond:ii},\eqref{cond:iii} are stable).
\end{thm}
\begin{proof}
This theorem is a simple generalization of Corollary \ref{cor:1} for the case that $\mathsf{P}(z):=\begin{bmatrix} zI-A & -G \\ C & H \end{bmatrix}$ is rank deficient for some $z \in Z_0 \subset \mathbb{C}$ but $|z| < 1$. For each such $z$, there exists $\begin{bmatrix} -x_z^\top & u_z^\top \end{bmatrix}^\top$ in the null space of $\mathsf{P}(z)$. It can be verified that the input sequence $u_k=z^k u_z$ and the initial state $x_z$ leads to the output is $y_k=0$ for all $k\geq 0$ but $x_k=z^k x_z$, where with a slight abuse of notation, $z^k$ represents the product of any permutations of $k$ numbers from $Z_0$. Since $|z|<1$ by assumption, $x_k \rightarrow 0$ as $k \rightarrow \infty$, which coincides with Definition \ref{def:det}.
%This theorem follows directly from Theorem \ref{thm:obsvTI}, Corollary \ref{cor:1} and Definition \ref{def:det}.
\end{proof}

\section{Algorithms for Minimum-variance Unbiased Filter for Simultaneous Input and State Estimation} \label{sec:filter}

For the filter design, we consider a recursive three-step filter\footnote{Note that the restriction to a recursive filter will be relaxed and shown to not lead to suboptimality in Theorem \ref{thm:globalULISE} for one of the filter variants.} as proposed in \cite{Gillijns.2007b,Yong.Zhu.Frazzoli.2013}, composed of an \emph{unknown input estimation} step which uses the current measurement and state estimate to estimate the unknown inputs in the best linear unbiased sense, a \emph{time update} step which propagates the state estimate based on the system dynamics, and a \emph{measurement update} step which updates the state estimate using the current measurement. Since this presents various options in terms of the order of execution of each step and the simulations in \cite{Yong.Zhu.Frazzoli.2013} appear to indicate the existence of two possible optimal structures, we propose two variants of a recursive three-step filter for the system described by \eqref{eq:sysX},\eqref{eq:y},\eqref{eq:sysY} to study both of these structures:
\begin{enumerate}[(I)]
\item \emph{Updated Linear Input \& State Estimator} (ULISE), which predicts $d_{1,k}$ using updated state estimate denoted by $\hat{x}_{k|k}$ \eqref{eq:variant1} as in \cite{Yong.Zhu.Frazzoli.2013},
\item \emph{Propagated Linear Input \& State Estimator}(PLISE), that uses propagated state estimate denoted by $\hat{x}^\star_{k|k}$ to predict $d_{1,k}$ \eqref{eq:variant2} as in \cite{Gillijns.2007b}.
\end{enumerate}

Given measurements up to time $k-1$, the three-step recursive filter\footnote{To initialize the filter, arbitrary initial values of $\hat{x}_{0|0}$, $P^x_0$ and $\hat{d}_{1,0}$ can be used since we will show that the ULISE filter is exponentially stable in Theorems \ref{thm:stableULISE} and \ref{thm:convULISE}, while the stability of the PLISE filter is shown in Theorem \ref{thm:convPLISE}. If $y_0$ and $u_0$ are available, we can find the minimum variance unbiased initial estimates given in the initialization of Algorithm \ref{algorithm1} using the linear minimum-variance-unbiased estimator \cite{Sayed.2003}.} can be summarized as follows:

\emph{Unknown Input Estimation}: 
\begin{subequations}
\begin{empheq}[left=\empheqlbrace]{align}
\hat{d}_{1,k}^{\rm \, I} &=M_{1,k} (z_{1,k}-C_{1,k} \hat{x}_{k|k}-D_{1,k} u_k) \label{eq:variant1} \\
\hat{d}_{1,k}^{\rm \, II} &=M_{1,k} (z_{1,k}-C_{1,k} \hat{x}^\star_{k|k}-D_{1,k} u_k) \label{eq:variant2} \quad \qquad \;
\end{empheq}
\end{subequations}
\begin{align}
\hat{d}_{2,k-1}&=M_{2,k} (z_{2,k}-C_{2,k} \hat{x}_{k|k-1}-D_{2,k} u_k) \label{eq:d2}\\
\hat{d}_{k-1}&= V_{1,k-1} \hat{d}_{1,k-1} + V_{2,k-1} \hat{d}_{2,k-1} \label{eq:d} 
\end{align}
\emph{Time Update}:
\begin{align}
\hspace{-0.1cm} \hat{x}_{k|k-1}&=A_{k-1} \hat{x}_{k-1 | k-1} + B_{k-1} u_{k-1} + G_{1,k-1} \hat{d}_{1,k-1} \label{eq:time} \\
\hat{x}^\star_{k|k}&=\hat{x}_{k|k-1}+G_{2,k-1} \hat{d}_{2,k-1} \label{eq:xstar}
\end{align}
\emph{Measurement Update}:
\begin{align}
\nonumber \hat{x}_{k|k}&=\hat{x}^\star_{k|k} +L_k(y_k-C_k \hat{x}^\star_{k|k}-D_k u_k) \\
&= \hat{x}^\star_{k|k} +\tilde{L}_k(z_{2,k}-C_{2,k} \hat{x}^\star_{k|k}-D_{2,k} u_k)  \quad \quad \label{eq:stateEst}
\end{align}
where $\hat{x}_{k-1|k-1}$, $\hat{d}_{1,k-1}$, $\hat{d}_{2,k-1}$ and $\hat{d}_{k-1}$ denote the optimal estimates of $x_{k-1}$, $d_{1,k-1}$, ${d}_{2,k-1}$ and $d_{k-1}$; 
$L_k \in \mathbb{R}^{n \times l}$, $\tilde{L}_k := L_k U_{2,k}\in \mathbb{R}^{n \times (l-p_{H_k})}$, $M_{1,k} \in \mathbb{R}^{p_{H_k} \times p_{H_k}}$ and $M_{2,k} \in \mathbb{R}^{(p-p_{H_k}) \times (l-p_{H_k})}$ are filter gain matrices that are chosen to minimize the state and input error covariances. 
Note that we applied $L_k=L_k U_{2,k} U_{2,k}^\top$ 
in \eqref{eq:stateEst},
which we will justify in Lemma \ref{lem:unbiased}.

The above recursive three-step filter represents a unified filter for simultaneously estimating unknown input and state for systems with an arbitrary direct feedthrough matrix, thus relaxing the assumptions on the direct feedthrough matrix in \cite{Gillijns.2007,Gillijns.2007b,Yong.Zhu.Frazzoli.2013}. By a suitable system transformation given in \eqref{eq:T_k}, the unknown input is decomposed into two components, $d_{1,k}$ and $d_{2,k}$; and similarly, the output equation into two orthogonal projections, $z_{1,k}$ and $z_{2,k}$, one with no direct feedthrough and the other with a full-rank feedthrough matrix. Hence, in a nutshell, the $d_{1,k}$ component of the unknown input can be estimated in the best linear unbiased sense by choosing $M_{1,k}$ as in \cite{Yong.Zhu.Frazzoli.2013,Gillijns.2007b} and the $d_{2,k}$ component by choosing $M_{2,k}$ as in \cite{Gillijns.2007}. On the other hand, the gain matrix $L_k$ is chosen to minimize the state estimate error covariance in an update similar to the Kalman filter. In fact, the proposed filter can be shown to be a generalization of the Kalman filter to systems with unknown inputs (see Section \ref{sec:kalman}). %and has an equivalent state estimation as \cite{Cheng.2009}.

\begin{algorithm}[!t] \small
\caption{ULISE algorithm}\label{algorithm1}
\begin{algorithmic}[1]
\State Initialize: $\hat{x}_{0|0}=\mathbb{E}[x_0]$; 
$P^x_{0|0}=\mathcal{P}^x_0$ ; 
$\hat{A}_{0}=A_{0}-G_{1,0}\Sigma_{0}^{-1} C_{1,0}$; $\hat{Q}_{0}=G_{1,0}\Sigma_{0}^{-1} R_{1,0}\Sigma_{0}^{-1} G_{1,0}^\top +Q_0$; $\hat{d}_{1,0}=\Sigma_0^{-1} (z_{1,0}-C_{1,0} \hat{x}_{0|0}-D_{1,0} u_0)$; $P^d_{1,0}=\Sigma_0^{-1}(C_{1,0} P^x_{0|0} C_{1,0}^\top+R_{1,0})\Sigma_0^{-1}$; 
 \For {$k =1$ to $N$}
\LineComment{Estimation of $d_{2,k-1}$ and $d_{k-1}$}
\State $\tilde{P}_k=\hat{A}_{k-1} P^x_{k-1|k-1} \hat{A}_{k-1}^\top +\hat{Q}_{k-1}$;
\State $\tilde{R}_{2,k}=C_{2,k} \tilde{P}_k C_{2,k}^\top+R_{2,k}$;
\State $P^d_{2,k-1}=(G_{2,k-1}^\top C_{2,k}^\top \tilde{R}_{2,k}^{-1} C_{2,k} G_{2,k-1})^{-1}$;
\State $M_{2,k}=P^d_{2,k-1} G_{2,k-1}^\top C_{2,k}^\top \tilde{R}^{-1}_{2,k}$;
\State $\hat{x}_{k|k-1}=A_{k-1} \hat{x}_{k-1|k-1}+B_{k-1} u_{k-1}+G_{1,k-1} \hat{d}_{1,k-1}$;
\State $\hat{d}_{2,k-1}=M_{2,k} (z_{2,k}-C_{2,k} \hat{x}_{k|k-1}-D_{2,k} u_k)$;
\State $\hat{d}_{k-1} =V_{1,k-1} \hat{d}_{1,k-1} + V_{2,k-1} \hat{d}_{2,k-1}$; 
\State $P^d_{12,k-1}=M_{1,k-1} C_{1,k-1} P^{x}_{k-1|k-1} A_{k-1}^\top C_{2,k}^\top M_{2,k}^\top$
\Statex \hspace{1.5cm} $-P^{d}_{1,k-1} G_{1,k-1}^\top C_{2,k}^\top M_{2,k}^\top$;
\State $P^d_{k-1}=V_{k-1} \begin{bmatrix} P^d_{1,k-1} & P^d_{12,k-1} \\ P^{d \top}_{12,k-1} & P^d_{2,k-1} \end{bmatrix} V_{k-1}^\top$;
\LineComment{Time update}
\State $\hat{x}^\star_{k|k}=\hat{x}_{k|k-1}+G_{2,k-1} \hat{d}_{2,k-1}$;
\State $P^{\star x}_{k|k}=G_{2,k-1} M_{2,k} R_{2,k} M_{2,k}^\top G_{2,k}^\top$
\Statex \hspace{1cm} $+(I-G_{2,k-1}M_{2,k}C_{2,k})\tilde{P}_k(I-G_{2,k-1}M_{2,k}C_{2,k})^\top$;
\State $\tilde{R}^\star_k=C_k P^{\star x}_{k|k} C_k^\top +R_k -C_k G_{2,k-1} M_{2,k} U_{2,k}^\top R_k$
\Statex \hspace{1.5cm}  $-R_k U_{2,k} M_{2,k}^\top G_{2,k-1}^\top C_k$;
\LineComment{Measurement update}
\State $K_k=P^{\star x}_{k|k} C_k^\top - G_{2,k-1} M_{2,k} U_{2,k}^\top R_k$; 
\State $M^\star_{1,k}:=\Sigma_k^{-1} (U_{1,k}^\top \tilde{R}^{\star \dagger}_k U_{1,k})^{-1} U_{1,k}^\top \tilde{R}^{\star \dagger}_k$; 
\State $L_k=K_k(I_l-U_{1,k} \Sigma_k M_{1,k}^\star)^\top \tilde{R}^{\star \dagger}_k$; 
\State $\hat{x}_{k|k}=\hat{x}^\star_{k|k}+L_k(y_k-C_k \hat{x}^\star_{k|k}-D_k u_k)$;
\State $P^x_{k|k}= (I-L_k C_k)G_{2,k-1}M_{2,k}U_{2,k}^\top R_{k}L_k^\top$
\Statex \hspace{1.5cm} $+ L_k R_{k} U_{2,k} M_{2,k}^\top G_{2,k-1}^\top (I-L_k C_k)^\top$
\Statex \hspace{1.5cm} $+(I-L_k C_k) P^{\star x}_{k|k} (I-L_k C_k)^\top+L_k R_k L_k^\top$;
\LineComment{Estimation of $d_{1,k}$}
\State $\tilde{R}_{1,k}=C_{1,k}P^x_{k|k}C_{1,k}^\top+R_{1,k}$; 
\State $M_{1,k}=\Sigma_k^{-1}$;
\State $P^d_{1,k}=M_{1,k} \tilde{R}_{1,k} M_{1,k}$;
\State $\hat{d}_{1,k}=M_{1,k} (z_{1,k}-C_{1,k} \hat{x}_{k|k}-D_{1,k} u_k)$;
\State $\hat{A}_{k}=A_{k}-G_{1,k}M_{1,k} C_{1,k}$;
\State $\hat{Q}_{k}=G_{1,k}M_{1,k}R_{1,k}M_{1,k}^\top G_{1,k}^\top +Q_k$;
\EndFor
\end{algorithmic}
\end{algorithm}

\begin{algorithm}[!h] \small
\caption{PLISE algorithm}\label{algorithm2}
\begin{algorithmic}[1]
\State Initialize: $\hat{x}_{0|0}=\mathbb{E}[x_0]$; $\hat{x}^\star_{0|0}=\mathbb{E}[x_0]$; 
$P^x_{0|0}=\mathcal{P}^x_0$ ; $P^{\star x}_{0|0}=\mathcal{P}^{ x}_0$ ; 
$\hat{A}_{0}=A_{0}-G_{1,0}\Sigma_{0}^{-1} C_{1,0}$; $\hat{Q}_{0}=G_{1,0}\Sigma_{0}^{-1} R_{1,0}\Sigma_{0}^{-1} G_{1,0}^\top +Q_0$; 
$\hat{d}_{1,0}=\Sigma_{0}^{-1} (z_{1,0}-C_{1,0} \hat{x}^\star_{0|0}-D_{1,0} u_0)$;
$P^d_{1,0}=\Sigma_0^{-1}(C_{1,0} P^{\star x}_{0|0} C_{1,0}^\top+R_{1,0})\Sigma_0^{-1}$; $P^{xd}_{1,0}=- P^{\star x}_{0|0} C_{1,0}^\top \Sigma_0^{-1}$; 
 \For {$k =1$ to $N$}
\LineComment{Estimation of $d_{2,k-1}$ and $d_{k-1}$}
\State $\tilde{P}_k=\hat{A}_{k-1} P^x_{k-1|k-1} \hat{A}_{k-1}^\top +\hat{Q}_{k-1}$;
\State $\tilde{R}_{2,k}=C_{2,k} \tilde{P}_k C_{2,k}^\top+R_{2,k}$;
\State $P^d_{2,k-1}=(G_{2,k-1}^\top C_{2,k}^\top \tilde{R}_{2,k}^{-1} C_{2,k} G_{2,k-1})^{-1}$;
\State $M_{2,k}=P^d_{2,k-1} G_{2,k-1}^\top C_{2,k}^\top \tilde{R}^{-1}_{2,k}$;
\State $\hat{x}_{k|k-1}=A_{k-1} \hat{x}_{k-1|k-1}+B_{k-1} u_{k-1} +G_{1,k-1} \hat{d}_{1,k-1}$;
\State $\hat{d}_{2,k-1}=M_{2,k} (z_{2,k}-C_{2,k} \hat{x}_{k|k-1}-D_{2,k} u_k)$;
\State $P^{xd}_{2,k-1}=-P^x_{k-1|k-1} A_{k-1}^\top C_{2,k}^\top M_{2,k}^{\top}$
\Statex \hspace{1.5cm} $-P^{xd}_{1,k-1} G_{1,k-1}^\top C_{2,k}^\top M_{2,k}^\top$;
\State $P^{d}_{12,k-1}=-P^{xd \; \top}_{1,k-1} A_{k-1}^\top C_{2,k}^\top M_{2,k}^{\top}$
\Statex \hspace{1.5cm}  $-P^{d}_{1,k-1} G_{1,k-1}^\top C_{2,k}^\top M_{2,k}^\top$;
\State $\hat{d}_{k-1} =V_{1,k-1} \hat{d}_{1,k-1} + V_{2,k-1} \hat{d}_{2,k-1}$; 
\State $P^d_{k-1}=V_{k-1} \begin{bmatrix} P^d_{1,k-1} & P^d_{12,k-1} \\ P^{d \top}_{12,k-1} & P^d_{2,k-1} \end{bmatrix} V_{k-1}^\top$;
\LineComment{Time update}
\State $\hat{x}^\star_{k|k}=\hat{x}_{k|k-1}+G_{2,k-1} \hat{d}_{2,k-1}$;
\State $P^{\star x}_{k|k}=\begin{bmatrix} A_{k-1} & G_{1,k-1} & G_{2,k-1} \end{bmatrix} $
\Statex \hspace{1.2cm} $\begin{bmatrix} P^x_{k-1|k-1} & P^{xd}_{1,k-1} & P^{xd}_{2,k-1} \\ P^{xd\; \top}_{1,k-1} & P^d_{1,k-1} & P^d_{12,k-1} \\ P^{xd\; \top}_{2,k-1} & P^{d \; \top}_{12,k-1} & P^d_{2,k-1} \end{bmatrix} \begin{bmatrix} A_{k-1}^\top \\ G_{1,k-1}^\top \\ G_{2,k-1}^\top \end{bmatrix}+Q_{k-1}$
\Statex \hspace{1.5cm} $-G_{2,k-1} M_{2,k} C_{2,k} Q_{k-1}-Q_{k-1} C_{2,k}^\top M_{2,k}^\top G_{2,k}^\top$;
\State $\tilde{R}^\star_k=C_k P^{\star x}_{k|k} C_k^\top +R_k -C_k G_{2,k-1} M_{2,k} U_{2,k}^\top R_k$
\Statex \hspace{1.5cm} $-R_k U_{2,k} M_{2,k}^\top G_{2,k-1}^\top C_k^\top$;
\LineComment{Estimation of $d_{1,k}$}
\State $\tilde{R}_{1,k}=C_{1,k}P^{\star \, x}_{k|k}C_{1,k}^\top+R_{1,k}$;
\State $M_{1,k}=\Sigma_k^{-1}$;
\State $P^d_{1,k}=M_{1,k} \tilde{R}_{1,k} M_{1,k}$;
\State $\hat{d}_{1,k}=M_{1,k} (z_{1,k}-C_{1,k} \hat{x}^\star_{k|k}-D_{1,k} u_k)$;
\State $\hat{A}_{k}=A_{k}-G_{1,k}M_{1,k} C_{1,k}$;
\State $\hat{Q}_{k}=G_{1,k}M_{1,k}R_{1,k}M_{1,k}^\top G_{1,k}^\top +Q_k$;
\LineComment{Measurement update}
\State $K_k=P^{\star x}_{k|k} C_k^\top - G_{2,k-1} M_{2,k} U_{2,k}^\top R_k$; 
\State $M^\star_{1,k}:=\Sigma_k^{-1} (U_{1,k}^\top \tilde{R}^{\star \dagger}_k U_{1,k})^{-1} U_{1,k}^\top \tilde{R}^{\star \dagger}_k$;
\State $L_k=K_k (I-U_{1,k} \Sigma_k M_{1,k}^\star)U_{2,k}^\top$;
\State $\hat{x}_{k|k}=\hat{x}^\star_{k|k}+L_k(y_k-C_k \hat{x}^\star_{k|k}-D_k u_k)$;
\State $P^x_{k|k}= L_k R_k L_k^\top+(I-L_k C_k)G_{2,k-1}M_{2,k}U_{2,k}^\top R_{k}L_k^\top$
\Statex \hspace{1.5cm} $+ L_k R_{k} U_{2,k} M_{2,k}^\top G_{2,k-1}^\top (I-L_k C_k)^\top$
\Statex \hspace{1.5cm} $+(I-L_k C_k) P^{\star x}_{k|k} (I-L_k C_k)^\top$;
\State $P^{xd}_{1,k}=-(I-L_k C_k)P^{\star x}_{k|k} C_{1,k}^\top M_{1,k}^\top$
\Statex \hspace{1.5cm} $- L_k R_k T_{2,k}^\top M_{2,k}^\top G_{2,k-1}^\top C_{1,k}^\top M_{1,k}^\top$;
\EndFor
\end{algorithmic}
\end{algorithm}

Note that the three steps are not given in the order of execution. 
In ULISE (see Algorithm \ref{algorithm1}), the estimation of $d_{2,k-1}$ is carried out before the time update, followed by the measurement update and finally, the estimation of $d_{1,k}$; while PLISE (see Algorithm \ref{algorithm2}) first computes $\hat{d}_{2,k-1}$, followed by the time update, the estimation of $d_{1,k}$ and the measurement update. Note also that Algorithms \ref{algorithm1} and \ref{algorithm2} for ULISE and PLISE are given with significant simplifications and a particular
choice of $\Gamma_k$ that will be further expounded in Section \ref{sec:analysis}. 

For both structures of the three-step filter variants, %provided the initial state estimate is unbiased, 
Algorithms \ref{algorithm1} and \ref{algorithm2} provide the `best' estimates of the states and unknown inputs in the minimum squared error sense, as given in the following lemma and will be proven in Section \ref{sec:analysis}. %, Both algorithms possess some nice properties, given by the following lemma and theorems which will be proven in Section \ref{sec:analysis}. 

\begin{lem} \label{lem:main}
Let the initial state estimate $\hat{x}_{0|0}$ be unbiased. If ${\rm rk}(C_{2,k} G_{2,k-1})=p-p_{H_{k-1}}$, then the ULISE and PLISE algorithms given in Algorithms \ref{algorithm1} and \ref{algorithm2} provide the unbiased, best linear estimate (BLUE) of the unknown input and the minimum-variance unbiased estimate of system states.
\end{lem}

In particular, we can show that ULISE is globally optimal over the class of linear state and input estimators. In other words, the structure of ULISE is optimal. Moreover,  the initial biases in the state and input estimates of ULISE decay exponentially if some conditions of uniform stabilizability and detectability are satisfied. Specifically for the time-invariant systems, conditions for the convergence of the error covariance matrix, $P^x_{k|k}$, as well as the filter gains, $L_k$, $M_{1,k}$ and $M_{2,k}$, to steady-state are provided.
To state these claims, which will be proven in Sections \ref{sec:optimality} and \ref{sec:stableULISE}, we first define: $\tilde{M}_{2,k}:=(C_{2,k} G_{2,k-1})^\dagger$, %(G_{2,k-1}^\top C_{2,k}^\top C_{2,k} G_{2,k-1})^{-1}G_{2,k-1}^\top C_{2,k}^\top $, 
$\hat{Q}_{k}=Q_k+G_{1,k}\Sigma_k^{-1}R_{1,k}\Sigma_k^{-1 \top} G_{1,k}^\top$, $\hat{A}_k=A_k-G_{1,k}M_{1,k} C_{1,k}$, $\tilde{A}_k:=(I-G_{2,k-1} \tilde{M}_{2,k} C_{2,k})\hat{A}_k+G_{2,k-1} \tilde{M}_{2,k} C_{2,k}$ and $\tilde{Q}_k=(I-G_{2,k-1} \tilde{M}_{2,k} C_{2,k})\hat{Q}_{k-1}(I-G_{2,k-1} \tilde{M}_{2,k} C_{2,k})^\top$.

\begin{thm}[Global Optimality of ULISE] \label{thm:globalULISE}
Let ${\rm rk}(C_{2,k} G_{2,k-1})=p-p_{H_{k-1}}$ and the initial state estimate $\hat{x}_{0|0}$ be unbiased. Then, the ULISE algorithm is globally optimal (over the class of all linear state and input estimators). 
\end{thm}

\begin{thm}[Stability of ULISE] \label{thm:stableULISE}
Suppose that ${\rm rk}(C_{2,k} G_{2,k-1})=p-p_{H_{k-1}}$. Then, that $(\tilde{A}_k,C_{2,k})$ is uniformly detectable\footnote{\label{footnote1} The notions of uniform detectability and stabilizability are standard (see, e.g., \cite[Section 2]{Anderson.Moore.1981}). A spectral test for these properties can be found in \cite{Peters.Iglesias.1999}.} is sufficient for the boundedness of the error covariance of the ULISE algorithm. Furthermore, if $(\tilde{A}_k,\tilde{Q}_k^{\frac{1}{2}})$ is uniformly stabilizable\footnoteref{footnote1}, ULISE is exponentially stable (i.e., its expected estimate errors decay exponentially).
\end{thm}

\begin{thm}[Convergence of ULISE to Steady-state] \label{thm:convULISE}
Let ${\rm rk}(C_{2,k} G_{2,k-1})=p-p_{H_{k-1}}$. Then, in the time-invariant case with $P^x_{0|0} \succ 0$, the filter gains of ULISE (exponentially) converge to a unique stationary solution if and only if
\vspace{-0.1cm}
\begin{enumerate}[(i)]
\item The linear time-invariant discrete-time system is strongly detectable, i.e Theorem \ref{thm:det} holds, and
\item \hspace{-0.2cm} ${\rm rk} \begin{bmatrix} \hat{A}- e^{j \omega}I & G_2  & \hat{Q}^{\frac{1}{2}} & 0 \\ e^{j \omega}C_2 & 0 & 0 & R_2^{\frac{1}{2}}\end{bmatrix}=n+l-p_H$, \\ $\forall \omega \in[0, 2 \pi]$
\end{enumerate}
where $\hat{Q}:=G_{1} M_{1} R_{1} M_{1}^\top G_{1}^\top$ $+Q$.
\end{thm}

On the other hand, although the structure of PLISE is suboptimal (as evidenced by the simulation examples in Section \ref{sec:examples}), PLISE does also possess stability guarantees in the time-invariant case, as stated in the following theorem and will be proven in Section \ref{sec:stablePLISE}.

\begin{thm}[Stability of PLISE (time-invariant)] \label{thm:convPLISE}
Let ${\rm rk}(C_{2,k} G_{2,k-1})=p-p_{H_{k-1}}$. Then, in the time-invariant case with $P^x_{0|0} \succ 0$, the estimate errors and error covariances of PLISE remain bounded if
\vspace{-0.2cm}
\begin{enumerate}[(i)]
\item The linear time-invariant discrete-time system is strongly detectable, i.e Theorem \ref{thm:det} holds, and \label{cond1}
\item ${\rm rk} \begin{bmatrix} e^{j \omega}I - F^s & Q^{s \frac{1}{2}} \end{bmatrix}=n$, $\forall \omega \in[0, 2 \pi]$ \label{cond2}
\end{enumerate}
where $F^s:=\hat{N} \hat{A} - \hat{S} \Theta^{-1} C_2$, $Q^s:=G_2 \tilde{M}_2 R_2 \tilde{M}^\top_2 G_2^\top +\hat{N} \hat{Q} \hat{N}^\top - \hat{S} \Theta^{-1} \hat{S}^\top$, $\hat{N}:=I-G_2 \tilde{M}_2 C_2$, $\tilde{M}_2:=(C_2 G_2)^\dagger$, $\hat{S}:=-\hat{N} \hat{A} G_2 \tilde{M}_2 R_2$, and assuming that $\Theta:=R_2 -C_2 G_2 \tilde{M}_2 R_2 -R_2 \tilde{M}_2^\top G_2^\top C_2^\top$ is invertible.
\end{thm}

Furthermore, these ULISE and PLISE algorithms reduce to filters in existing literature, as shown in Section \ref{sec:connection}.

\begin{rem}
The stability (and convergence to steady-state in the time-invariant case) of both variants of the unified state and input estimator 
 is closely related to the \textsf{strong detectability} of the system. In the time-varying case, the sufficient condition of uniform detectability of ULISE implies strong detectability (cf. Definition \ref{def:det} and \cite[Lemma 2.2]{Anderson.Moore.1981}) whereas in the time-invariant case, the strong detectability condition appears explicitly for the stability of both ULISE and PLISE. On the other hand, uniform stabilizability of Theorem \ref{thm:stableULISE} parallels the sufficient condition for the Kalman filter and Condition (ii) of Theorems \ref{thm:convULISE} and \ref{thm:convPLISE} corresponds to the \textsf{controllability} of the filter dynamics \textsf{on the unit circle}, akin to the %requirement of 
 system \textsf{controllability on the unit circle} for the Kalman filter. 
Conversely, if the system is not strongly detectable, then it is not possible to obtain unbiased estimates of the states and unknown inputs even for the case with no noise.
%For both variants of the unified state and input estimator, the convergence and stability (i.e., boundedness of the error covariances)
% is closely related to the \textsf{strong detectability} of the system. Note the similarity of this condition to the convergence and  stability conditions for the Kalman filter, which, in turn, is related to the \textsf{detectability} of the system. On the other hand, the second condition of Theorems \ref{thm:convULISE} and \ref{thm:convPLISE} corresponds to the \textsf{controllability} of the filter dynamics \textsf{on the unit circle}, akin to the requirement of system \textsf{controllability on the unit circle} for the Kalman filter.
\end{rem}

\vspace{-0.1cm}
\section{Filter Description and Analysis} \label{sec:analysis}
\vspace{-0.1cm}
For the analysis of the proposed filter,
let $\tilde{x}_{k|k}:=x_k-\hat{x}_{k|k}$, $\tilde{x}^\star_{k|k}:=x_k-\hat{x}^\star_{k|k}$, $\tilde{d}_{1,k}:=d_{1,k}-\hat{d}_{1,k}$, $\tilde{d}_{2,k}:=d_{2,k}-\hat{d}_{2,k}$, $\tilde{d}_{k}:=d_{k}-\hat{d}_{k}$, $P^x_{k|k}:=\mathbb{E}[\tilde{x}_{k|k} \tilde{x}_{k|k}^\top]$, $P^{\star x}_{k|k}:=\mathbb{E}[\tilde{x}^\star_{k|k} \tilde{x}_{k|k}^{\star \top}]$, $P^d_{1,k} :=\mathbb{E}[\tilde{d}_{1,k} \tilde{d}_{1,k}^\top]$, $P^d_{2,k} :=\mathbb{E}[\tilde{d}_{2,k} \tilde{d}_{2,k}^\top]$, $P^d_{12,k} =(P^d_{21,k})^\top :=\mathbb{E}[\tilde{d}_{1,k} \tilde{d}_{2,k}^\top]$, $P^{xd}_{1,k}=(P^{xd}_{1,k})^\top:=\mathbb{E}[\tilde{x}_{k|k}\tilde{d}_{1,k}^\top]$, $P^{xd}_{2,k}=(P^{xd}_{2,k})^\top:=\mathbb{E}[\tilde{x}_{k|k}\tilde{d}_{2,k}^\top]$ and $P^d_{k} :=\mathbb{E}[\tilde{d}_{k} \tilde{d}_{k}^\top]$. We initially assume that the initial state estimate is unbiased, i.e., $\mathbb{E}[\hat{x}_{0|0}]=\mathbb{E}[\hat{x}^\star_{0|0}]=\mathbb{E}[x_0]$ and present a lemma that summarizes the unbiasedness of the state and unknown input estimates for all time steps that is one piece of the claim in Lemma \ref{lem:main}.

\begin{lem} \label{lem:unbiased}
Let $\hat{x}_{0|0}=\hat{x}^\star_{0|0}$ be unbiased, then the input and state estimates, $\hat{d}_{k-1}$, $\hat{x}^\star_{k|k}$ and $\hat{x}_{k|k}$, are unbiased for all $k$, if and only if $M_{1,k} \Sigma_k=I$, $M_{2,k}C_{2,k} G_{2,k-1}=I$ and $L_k U_{1,k}=0$. Consequently, ${\rm rk}(C_{2,k} G_{2,k-1})=p-p_{H_{k-1}}$ and $L_k=L_k U_{2,k} U_{2,k}^\top$.
\end{lem}
\begin{proof}
We observe from \eqref{eq:sysY}, \eqref{eq:variant1}, \eqref{eq:variant2} and \eqref{eq:d2} that
\begin{subequations} \label{eq:d1_hat}
\begin{empheq}[left=\empheqlbrace]{align}
\hat{d}_{1,k}^{\rm \, I} &=M_{1,k} (C_{1,k} \tilde{x}_{k|k}+\Sigma_k d_{1,k}+v_{1,k}) \\ 
\hat{d}_{1,k}^{\rm \, II} &=M_{1,k} (C_{1,k} \tilde{x}^\star_{k|k}+\Sigma_k d_{1,k}+v_{1,k}) \; 
\end{empheq}
\end{subequations} \vspace{-0.5cm}
\begin{align} \label{eq:d2_hat}
\nonumber \hat{d}_{2,k-1}=& M_{2,k} (C_{2,k} (A_{k-1}\tilde{x}_{k-1|k-1}+G_{1,k-1} \tilde{d}_{1,k-1})\\
& +w_{k-1}+v_{2,k}+C_{2,k} G_{2,k-1} d_{2,k-1}). 
\end{align} 
On the other hand,  from \eqref{eq:time} and \eqref{eq:xstar}, the error in the propagated state estimate can be obtained as:
\begin{align}
\nonumber \tilde{x}^\star_{k|k} =& A_{k-1} \tilde{x}_{k-1|k-1}+G_{1,k-1} \tilde{d}_{1,k-1}+G_{2,k-1} \tilde{d}_{2,k-1}\\ &+w_{k-1}. \label{eq:x_kk-1}
\end{align}
Moreover, from \eqref{eq:y} and \eqref{eq:stateEst}, the updated state estimate error is
\begin{align} \label{eq:x_kk}
\tilde{x}_{k|k}=(I-L_kC_k) \tilde{x}^\star_{k|k}-L_k U_{1,k} \Sigma_k d_{1,k} - L_k v_k.
\end{align}
We show by induction that the estimates $\hat{d}_{k}$, $\hat{x}_{k|k}$ and $\hat{x}^\star_{k|k}$ are unbiased. 
For the base case, since $\hat{x}_{0|0}$ and $\hat{x}^\star_{0|0}$ are unbiased and the process and measurement noise are assumed to have zero mean, $\mathbb{E}[w_0]=0$, $\mathbb{E}[v_0]=0$, from \eqref{eq:d1_hat} and \eqref{eq:d2_hat}, $\mathbb{E}[\hat{d}_{1,0}]=d_{1,0}$ and $\mathbb{E}[\hat{d}_{2,0}]=d_{2,0}$, i.e., $\hat{d}_{1,0}$ and $\hat{d}_{2,0}$ are unbiased, if and only if $M_{1,0} \Sigma_0=I$, and $M_{2,1}C_{2,1} G_{2,0}=I$. Hence, $\hat{d}_0$ is unbiased.
In the inductive step, we assume that $\mathbb{E}[\tilde{x}_{k-1|k-1}]=\mathbb{E}[\tilde{x}^\star_{k-1|k-1}]=0$. Then, the input estimates are unbiased, i.e., $\mathbb{E}[\tilde{d}_{k-1}]=\mathbb{E}[\tilde{d}_{1,k-1}]=\mathbb{E}[\tilde{d}_{2,k-1}]=0$, if and only if $M_{1,k-1} \Sigma_{k-1}=I$, and $M_{2,k}C_{2,k} G_{2,k-1}=I$. Since the process noise has zero mean, by \eqref{eq:x_kk-1}, $\mathbb{E}[\tilde{x}^\star_{k|k}]=0$. Similarly, from \eqref{eq:x_kk} with a zero-mean measurement noise, we impose the constraint $L_k U_{1,k}=0$ such that we obtain $\mathbb{E}[\tilde{x}_{k|k}]=0$. Therefore, by induction, $\mathbb{E}[\tilde{x}^\star_{k|k}]=0$  and $\mathbb{E}[\tilde{x}_{k|k}]=0$ for all $k$. Since we require $M_{2,k}C_{2,k} G_{2,k-1}=I$ for all $k$ for the existence of an unbiased input estimate, it follows that ${\rm rk}(C_{2,k}G_{2,k-1})=p-p_{H_k-1}$ is a necessary and sufficient condition. Furthermore, $L_k=L_k U_k U_k^\top = L_k U_{2,k} U_{2,k}^\top$ since $L_k U_{1,k}=0$.
\end{proof}

\begin{rem}
The assumption of an unbiased initial state is common in existing filters, including the Kalman filter, although this is not critical because the resulting state error dynamics is a stable linear system and the effect of an initial state error decays exponentially.
\end{rem}

Next, we continue the proof of Lemma \ref{lem:main} in three subsections, one for each step of the three-step recursive filter. Then, the subsequent two subsections present the proof of Theorems \ref{thm:globalULISE}, \ref{thm:convULISE}, and \ref{thm:convPLISE}.

\subsection{Unknown Input Estimation} \label{sec:inputEst}

To obtain an optimal estimate of $\hat{d}_{k-1}$ using \eqref{eq:d}, we estimate both components of the unknown input as the best linear unbiased estimates (BLUE). This means that the expected input estimate is unbiased, i.e., $\mathbb{E}[\hat{d}_{1,k}]=d_{1,k}$, $\mathbb{E}[\hat{d}_{2,k}]=d_{2,k}$ and $\mathbb{E}[\hat{d}_{k}]=d_{k}$, as was shown in Lemma \ref{lem:unbiased}, and that the mean squared error of the estimate is the lowest possible, shown next in Theorem \ref{thm:mse}.
%To have an optimal estimate of $\hat{d}_{k-1}$ using \eqref{eq:d}, we need to estimate both components of the unknown input such that they are the best linear unbiased estimates (BLUE). This means that the expected input estimate must be unbiased, i.e., $\mathbb{E}[\hat{d}_{1,k}]=d_{1,k}$ and $\mathbb{E}[\hat{d}_{2,k}]=d_{2,k}$, which is shown in Lemma \ref{lem:unbiased}, and that the mean squared error of the estimate is the lowest possible, shown in Lemma \ref{thm:mse}.

\begin{thm} \label{thm:mse}
Suppose $\hat{x}_{0|0}=\hat{x}^\star_{0|0}$ are unbiased.  
Then \eqref{eq:variant1}, \eqref{eq:variant2} and \eqref{eq:d2} provide the best linear input estimate (BLUE) with $M_{1,k}$ and $M_{2,k}$ given by 
\begin{align}
\hspace{-0.25cm}  M_{1,k}&=\Sigma_k^{-1} \label{eq:Mk}\\
\hspace{-0.25cm} M_{2,k}&=(G_{2,k-1}^\top C_{2,k}^\top \tilde{R}_{2,k}^{-1} C_{2,k} G_{2,k-1})^{-1}G_{2,k-1}^\top C_{2,k}^\top \tilde{R}_{2,k}^{ -1} \hspace{-0.15cm} \label{eq:Mk2}
\end{align}
while the covariance matrices of the optimal input error estimates are 
\begin{align}
 P^d_{1,k}&=\Sigma_k^{-1} \tilde{R}_{1,k} \Sigma^{-1}_k  \\
P^d_{2,k-1}&=(G_{2,k-1}^\top C_{2,k}^\top \tilde{R}_{2,k}^{-1} C_{2,k} G_{2,k-1})^{-1}
\end{align} \vspace{-0.5cm}
\begin{subequations}
\begin{empheq}[left=\empheqlbrace]{align}
\nonumber \tilde{R}_{1,k}^{\rm \, I} :=& \mathbb{E}[e^{\rm \, I}_{1,k} e^{\rm \, I \, \top}_{1,k}]
=C_{1,k} P^x_{k|k}C_{1,k}^\top + R_{1,k} \\&  \quad - C_{1,k} L_k R_k T_{1,k}^\top
 -T_{1,k}R_kL_k^\top C_{1,k}^\top \label{eq:Rvariant1} \\
\tilde{R}_{1,k}^{\rm \, II} :=& \mathbb{E}[e^{\rm \, II}_{1,k} e^{\rm \, II \, \top}_{1,k}]=C_{1,k} P^{\star \, x}_{k|k}C_{1,k}^\top + R_{1,k}  \label{eq:Rvariant2}
\end{empheq}
\end{subequations} \vspace{-0.5cm}
\begin{align}
\tilde{R}_{2,k}:=& \mathbb{E}[e_{2,k} e_{2,k}^\top]=C_{2,k} \tilde{P}_{k} C_{2,k}^\top + R_{2,k}  \label{eq:Rd2}
\end{align}
where $\tilde{P}_k:=\hat{A}_{k-1} P^x_{k-1|k-1} \hat{A}_{k-1}^\top +\hat{Q}_{k-1}$, $\hat{A}_k:=A_k-G_{1,k}M_{1,k}C_{1,k}$ and $\hat{Q}_{k}:=Q_k+G_{1,k} M_{1,k} R_{1,k} M_{1,k}^\top G_{1,k}^\top$.
\end{thm}
\begin{proof}
Let $\tilde{z}^{\rm \, I}_{1,k} := z_{1,k}-C_{1,k} \hat{x}_{k|k}-D_{1,k} u_k$, $\tilde{z}^{\rm \, II}_{1,k} := z_{1,k}-C_{1,k} \hat{x}^\star_{k|k}-D_{1,k} u_k$ and $\tilde{z}_{2,k} := z_{2,k}-C_{2,k} \hat{x}_{k|k-1}-D_{2,k} u_k$. Then, we have
\begin{subequations} \label{eq:ztilde1}
\begin{empheq}[left=\empheqlbrace]{align}
\tilde{z}^{\rm \, I}_{1,k} &= \Sigma_k d_{1,k} + e^{\rm \, I}_{1,k}, \\
\tilde{z}^{\rm \, II}_{1,k} &=  \Sigma_k d_{1,k} + e^{\rm \, II}_{1,k},
\end{empheq}
\end{subequations} \vspace{-0.5cm}
\begin{align}
\tilde{z}_{2,k} &= C_{2,k} G_{2,k-1} d_{2,k-1}+e_{2,k}, \label{eq:ztilde2}
\end{align}
where $e^{\rm \, I}_{1,k}:=C_{1,k} \tilde{x}_{k|k} + v_{1,k}$, $e^{\rm \, II}_{1,k}:=C_{1,k} \tilde{x}^\star_{k|k}+ v_{1,k}$ and $e_{2,k}:=C_{2,k} (A_{k-1} \tilde{x}_{k-1|k-1} +G_{1,k-1} \tilde{d}_{1,k-1}+w_{k-1}) +v_{2,k}$. From the unbiasedness of the state and input estimates (Lemma \ref{lem:unbiased}), $\mathbb{E}[e^{\rm \, I}_{1,k}]=0$, $\mathbb{E}[e^{\rm \, II}_{1,k}]=0$ and $\mathbb{E}[e_{2,k}]=0$. Their covariance matrices are given by 
\begin{subequations}
\begin{empheq}[left=\empheqlbrace]{align}
 \tilde{R}^{\rm \, I}_{1,k} :=& \mathbb{E}[e^{\rm \, I}_{1,k} e^{\rm \, I \top}_{1,k}]
=C_{1,k} P^x_{k|k}C_{1,k}^\top +R_{1,k} \\ 
\tilde{R}^{\rm \, II}_{1,k} :=& \mathbb{E}[e^{\rm \, II}_{1,k} e^{\rm \, II \top}_{1,k}]=C_{1,k} P^{\star \, x}_{k|k}C_{1,k}^\top +R_{1,k}
\end{empheq}
\end{subequations} \vspace{-0.5cm}
\begin{align}
\tilde{R}_{2,k} :=& \mathbb{E}[e_{2,k} e_{2,k}^\top]=C_{2,k} \tilde{P}_{k}C_{2,k}^\top +R_{2,k}
\end{align}
where the simplified expressions above is obtained by applying $\mathbb{E}[\tilde{x}_{k|k} v_{1,k}^\top]=(\mathbb{E}[v_{1,k} \tilde{x}_{k|k}^\top])^\top=0$, $\mathbb{E}[\tilde{x}_{k|k} w_k^\top]=0$, $\mathbb{E}[\tilde{d}_{1,k} w_k^\top]=0$, $\mathbb{E}[\tilde{x}_{k-1|k-1} v_{2,k}^\top]=0$, $\mathbb{E}[\tilde{d}_{1,k-1} v_{2,k}^\top]=0$ and $\mathbb{E}[{w}_{k-1} v_{2,k}^\top]=0$, as well as $L_k=L_{k} U_{2,k} U_{2,k}^\top$ from Lemma \ref{lem:unbiased} and \eqref{eq:R12} to obtain $L_k R_k T_{1,k}^\top=L_k U_{2,k} U_{2,k}^\top R_k T_{1,k}^\top=L_k U_{2,k} R_{21,k}=0$.
Next, we obtain the estimates for $\hat{d}_{1,k}$ and $\hat{d}_{2,k}$ given by \eqref{eq:variant1},  \eqref{eq:d2}, \eqref{eq:Mk} and \eqref{eq:Mk2} by applying the well known generalized least squares (GLS) estimate (see, e.g., \cite[Theorem 3.1.1]{Sayed.2003}), which are linear minimum-variance unbiased estimates, a.k.a. as best linear unbiased estimates (BLUE). Note that since $\Sigma_k$ is invertible, there is one unique unbiased estimate of $\hat{d}_{1,k}$.
%Following the estimation approach outlined in \cite[pp. 96-98]{kailath.2000} such that  the Gauss-Markov Theorem assumption is satisfied, we obtain the best linear unbiased estimates (BLUE) for $\hat{d}_{1,k}^{\rm \, I}$, $\hat{d}_{1,k}^{\rm \, II}$ and $\hat{d}_{2,k}$ given by \eqref{eq:variant1}, \eqref{eq:variant2}, \eqref{eq:d2}, \eqref{eq:Mk} and \eqref{eq:Mk2}.
Since $M_{1,k} \Sigma_k=I$ and $M_{2,k} C_{2,k} G_{2,k-1}=I$, the input estimate errors, and their covariance matrices are as follows
\begin{align} \label{eq:dtilde}
\nonumber \tilde{d}_{1,k}^{\rm \, I} &=-M_{1,k} e_{1,k}^{\rm \, I},  \tilde{d}_{1,k}^{\rm \, II} =-M_{1,k} e_{1,k}^{\rm \, II}, \\
\nonumber\tilde{d}_{2,k-1} &=-M_{2,k} e_{2,k} \\
\nonumber P_{1,k}^d&=\mathbb{E}[\tilde{d}_{1,k} \tilde{d}_{1,k}^\top]=M_{1,k} \mathbb{E}[e_{1,k} e_{1,k}^\top] M_{1,k}^\top \\
&=\Sigma_k^{-1} \tilde{R}_{1,k} \Sigma_k^{-1}\\
\nonumber  P_{2,k-1}^d&=\mathbb{E}[\tilde{d}_{2,k-1} \tilde{d}_{2,k-1}^\top]=M_{2,k} \mathbb{E}[e_{2,k} e_{2,k}^\top] M_{2,k}^\top  \\
\nonumber &=(G_{2,k-1}^\top C_{2,k}^\top \tilde{R}_{2,k}^{-1} C_{2,k} G_{2,k-1})^{-1}
\end{align} 
Finally, we note the following equality:
\begin{align}
&{\rm tr} (\mathbb{E}[\tilde{d}_k \tilde{d}_k^\top])={\rm tr} (\mathbb{E}[V_k \begin{bmatrix} \tilde{d}_{1,k} \\ \tilde{d}_{2,k} \end{bmatrix} \begin{bmatrix} \tilde{d}_{1,k} & \tilde{d}_{2,k} \end{bmatrix} V_k^\top]) \label{eq:traced}\\
\nonumber  &= {\rm tr} (V_k^\top V_k \mathbb{E}[\begin{bmatrix} \tilde{d}_{1,k} \\ \tilde{d}_{2,k} \end{bmatrix} \begin{bmatrix} \tilde{d}_{1,k} & \tilde{d}_{2,k} \end{bmatrix}])={\rm tr} (P^d_{1,k})+{\rm tr} (P^d_{2,k}).
\end{align} 
Since the unbiased estimate of $\hat{d}_{1,k}$ is unique, we have $\min {\rm tr} (\mathbb{E}[\tilde{d}_k \tilde{d}_k^\top])= {\rm tr} (\mathbb{E}[\tilde{d}_{1,k} \tilde{d}_{1,k}^\top]) + \min {\rm tr} (\mathbb{E}[\tilde{d}_{2,k} \tilde{d}_{2,k}^\top])$, from which it can be observed that the unbiased estimate $\hat{d}_k$ has minimum variance when $\hat{d}_{1,k}$ and $\hat{d}_{2,k}$ have minimum variances.
\end{proof}

\begin{rem}
Moreover, if $w_{i,k}$ and $v_{i,k}$ for $i=\{1,2\}$ are white Gaussian noises, which lead to $e_{i,k}$ being white and Gaussian, then \eqref{eq:variant1}, \eqref{eq:variant2}, \eqref{eq:d2} and \eqref{eq:d} also provide the minimum variance unbiased (MVU) input estimate.
\end{rem}

\subsection{Time Update} \label{sec:timeUpdate}

The time update is given by \eqref{eq:time} and \eqref{eq:xstar}, and the error in the propagated state estimate by \eqref{eq:x_kk-1} and its covariance matrix are given by

\vspace{-0.5cm}
\small\begin{align}
\nonumber P^{\star x}_{k|k}=&\begin{bmatrix} A_{k-1}^\top \\ G_{1,k-1}^\top \\ G_{2,k-1}^\top \end{bmatrix}^\top \begin{bmatrix} P^x_{k-1|k-1} & P^{xd}_{1,k-1} & P^{xd}_{2,k-1} \\ P^{xd\; \top}_{1,k-1} & P^d_{1,k-1} & P^d_{12,k-1} \\ P^{xd\; \top}_{2,k-1} & P^{d \; \top}_{12,k-1} & P^d_{2,k-1} \end{bmatrix} \begin{bmatrix} A_{k-1}^\top \\ G_{1,k-1}^\top \\ G_{2,k-1}^\top \end{bmatrix}\\ & \hspace{-0.9cm}+Q_{k-1}-G_{2,k-1} M_{2,k} C_{2,k} Q_{k-1} - Q_{k-1} C_{2,k}^\top M_{2,k}^\top G_{2,k-1}^\top \label{eq:Pstar}.
\end{align}
\normalsize
Using \eqref{eq:dtilde}, \eqref{eq:Mk} and \eqref{eq:Mk2}, \eqref{eq:Pstar} can be rewritten as
\begin{align}
\nonumber P^{\star x}_{k|k}=&(I-G_{2,k-1} M_{2,k} C_{2,k}) \mathring{P}_k (I-G_{2,k-1} M_{2,k} C_{2,k})^\top \\
&+G_{2,k-1} M_{2,k} R_{2,k} M_{2,k}^\top G_{2,k-1}^\top \label{eq:Pstar2}
\end{align}
where we applied $L_k=L_k U_{2,k} U_{2,k}^\top$ from Lemma \ref{lem:unbiased} and $T_{1,k} R_k T_{2,k}^\top=0$ from \eqref{eq:R12}, and defined $\mathring{A}_{k}:=\hat{A}_{k}-A_{k} L_{k} C_{k}$, $\mathring{G}_{k}:=G_{1,k} M_{1,k}T_{1,k}+A_{k} L_{k}$,  
 and $\mathring{P}_{k}$ as follows:
%\begin{subequations}
\begin{empheq}[left=\empheqlbrace]{align} \label{eq:Pring}
 \mathring{P}^{{\rm \, I}}_{k} :=&\tilde{P}_k \\ 
\nonumber \hspace{0.2cm} \mathring{P}^{{\rm \, II}}_{k} :=&\mathring{A}_{k-1} P^{\star x}_{k-1|k-1} \mathring{A}_{k-1}^\top +{Q}_{k-1}+\mathring{G}_{k-1} R_{k-1} \mathring{G}_{k-1}^\top \\
 \nonumber & + \mathring{A}_{k-1} G_{2,k-2} M_{2,k-1} U_{2,k-1}^\top R_{k-1} L_{k-1}^\top A_{k-1}^\top\\
 \nonumber & +A_{k-1} L_{k-1} R_{k-1} U_{2,k-1} M_{2,k-1}^\top G_{2,k-2}^\top \mathring{A}_{k-1}^\top. 
\end{empheq}\vspace{-1.8cm}
%\end{subequations}
%and applied $L_k=L_k U_{2,k} U_{2,k}^\top$ from Lemma \ref{lem:unbiased} and $T_{1,k} R_k T_{2,k}^\top=0$ from \eqref{eq:R12}.

\subsection{Measurement Update} \label{sec:measUpdate}

In the measurement update step, the measurement $y_k$ is used to update the propagated estimate of $\hat{x}^\star_{k|k}$ and $P^{\star x}_{k|k}$. From \eqref{eq:y} and \eqref{eq:stateEst}, the updated state estimate error is given by \eqref{eq:x_kk}
where the constraint $L_k U_{1,k}=0$ (Lemma \ref{lem:unbiased}) must be imposed for all $k$ such that the state estimate is unbiased ($\mathbb{E}[\tilde{x}_{k|k}]=0$) for all possible $d_{1,k}$, since $\Sigma_k$ has full rank.  Note that the residual/innovations term in the measurement update step given in \eqref{eq:stateEst} appears to not contain an $H_k \hat{d}_k$ term as would be expected. This term is actually present, but has been nullified by the unbiasedness constraint (Lemma \ref{lem:unbiased}), since $L_k H_k =L_k U_{1,k} \Sigma_k V_{1,k}^\top=0$. This is also in line with the practical reason that the unknown input estimate is not yet available.
Next, the covariance matrix of the state error is computed as
\begin{align}
\nonumber P^x_{k|k}=&(I-L_k C_k) P^{\star x}_{k|k} (I-L_k C_k)^\top + L_k R_k L_k^\top \\
\nonumber &+(I-L_k C_k) G_{2,k-1}M_{2,k} U_{2,k}^\top R_k L_k^\top  \\
\nonumber &+L_k R_k U_{2,k} M_{2,k}^\top G_{2,k-1}^\top (I-L_k C_k)^\top \\
:=& P^{\star x}_{k|k} + {L}_k \tilde{R}^\star_k {L}_k -{L}_k {S}_k^\top-{S}_k {L}_k^\top \label{eq:cov}
\end{align}
where $\mathbb{E}[\tilde{x}^\star _{k|k} v_k^\top]=-G_{2,k-1} M_{2,k} U_{2,k}^\top R_k$, and we defined $\tilde{R}^\star_k:= C_{k} P^{\star x}_{k|k} C_{k}^\top +R_{k} -C_{k} G_{2,k-1} M_{2,k} U_{2,k}^\top R_{k}-R_{k} U_{2,k} M_{2,k}^\top G_{2,k-1}^\top C_{k}^\top$ and ${S}_k:= -G_{2,k-1} M_{2,k} U_{2,k}^\top R_{k}+P^{\star x}_{k|k} C_{k}^\top$.
Using \eqref{eq:Pstar2}, we can rewrite the expression $\tilde{R}^\star_k=N_k \hat{R}_k N_k^\top$ where $\hat{R}_k:=C_k \mathring{P}_k C_k^\top +R_k$, $N_k:=I-C_k G_{2,k-1} M_{2,k} U_{2,k}^\top$ and $ \mathring{P}_k$ as defined in \eqref{eq:Pring}.

To obtain an unbiased minimum variance estimator, we then proceed to derive the optimal gain matrix $L_k$, by minimizing the trace of \eqref{eq:cov}, since the trace represents the sum of the estimation error variances of the states, subject to the constraint $L_k U_{1,k} =0$. However, the next lemma shows that $\tilde{R}^\star_k=N_k \hat{R}_k N_k^\top$ is singular because $N_k$ is rank deficient, except when $p=p_H$, i.e., ${H_k}$ has full rank.

\begin{lem}
Consider $M_{2,k}$ that satisfies \eqref{eq:Mk2}, then 
$N_k$ has rank $p_R=l-p+p_{H_{k-1}}$ and $p_{H_{k-1}} \leq p_R \leq l$.
\end{lem}
\begin{proof}
Since $M_{2,k}$ satisfies \eqref{eq:Mk2}, 
$N_k$ is an idempotent matrix, i.e., $N_k N_k= N_k$. From \cite[Fact 3.12.9 and Proposition 2.6.3]{Bernstein.2009} and ${\rm rk}(C_{2,k} G_{2,k-1})=p-p_{H_{k-1}}$, we obtain $p_R:={\rm rk}(I_l-C_k G_{2,k-1} M_{2,k} U_{2,k}) =l- {\rm rk} (C_k G_{2,k-1} M_{2,k} U_{2,k}) = l-p +p_{H_{k-1}} \leq l$. Since we assumed $l \geq p$, we have $p_{H_{k-1}} \leq p_R \leq l$.
\end{proof}

Hence, the optimal gain matrix $L_k$ is in general not unique. Similar to \cite{Gillijns.2007}, we propose a gain matrix $L_k$ of the form
$L_k=\overline{L}_k \Gamma_k$
%\end{align}
where $\Gamma_k \in \mathbb{R}^{p_R \times l}$ 
is an arbitrary matrix which has to be chosen such that $\Gamma_k \tilde{R}^\star_k \Gamma_k^\top$ has full rank. With this, we compute the optimal gain $\overline{L}_k$ and thus $L_k$ in the following theorem.

\begin{thm} \label{thm:mvu_state}
Suppose $\hat{x}_{0|0}=\hat{x}^\star_{0|0}$ are unbiased, 
and let $\Gamma_k \in \mathbb{R}^{p_R \times l}$ be chosen such that $\Gamma_k \tilde{R}^\star_k \Gamma_k^\top$ has full rank. Then, the minimum-variance unbiased state estimator is obtained with the gain matrix $L_k$ given by 
\begin{align}
L_k=K_k \check{R}_k (I_l-H_{1,k} M_{1,k}^\star) =K_k  (I_l-H_{1,k} M_{1,k}^\star)^\top \check{R}_k \hspace{-0.1cm}\label{eq:Lk}
\end{align}
where $H_{1,k}=U_{1,k} \Sigma_k$, $M^\star_{1,k}:=\Sigma_k^{-1} (U_{1,k}^\top \check{R}_k U_{1,k})^{-1} U_{1,k}^\top \check{R}_k$, $\check{R}_k:=\Gamma_k^\top (\Gamma_k \tilde{R}_{k}^\star \Gamma_k^\top)^{-1} \Gamma_k$, and
\begin{align*}
K_k:=& (P^{\star x}_{k|k} C_k^\top - G_{2,k-1} M_{2,k} U_{2,k}^\top R_k)\\
=& (\mathring{P}_k C_k^\top - G_{2,k-1} M_{2,k} U_{2,k}^\top \hat{R}_k)N_k^\top,
\end{align*}
with $M_{2,k}$ and $\tilde{P}_k$ as defined in the Theorem \ref{thm:mse}, and $\hat{R}_k$ and $\tilde{R}^\star_k$ as defined in the text following \eqref{eq:cov}.
\end{thm}
\begin{proof}
By Lemma \ref{lem:unbiased}, the state estimates are unbiased.
Next, we employ the optimization approach with Lagrange multipliers ($\Lambda_k \in \mathbb{R}^{n \times p_H}$) in \cite{Kitanidis.1987,Gillijns.2007b,Yong.Zhu.Frazzoli.2013}, to find the particular gain $L_k$ that minimizes the trace of of the covariance matrix $P^x_{k|k}$, while being subjected to the constraint $L_k U_{1,k}=0$ which is a necessary condition for obtaining an unbiased estimate. This constrained optimization problem can be solved using differential calculus with the Lagrangian given by
\begin{align*}
\mathcal{L}(\overline{L}_k,\Lambda_k):={\rm tr}(P^x_{k|k})-2 \; {\rm tr}(\overline{L}_k \Gamma_k U_{1,k} \Lambda_k^\top)
\end{align*}
with a filter gain of the form $L_k=\overline{L}_k \Gamma_k$.
Differentiating the Lagrangian with respect to $\overline{L}_k$ and $\Lambda_k$, and setting it to zero, we obtain
\begin{align*}
\frac{\partial{\mathcal{L}}}{\partial \overline{L}_k} &= 2(\Gamma_k \tilde{R}^\star_k \Gamma_k^\top \overline{L}_k^\top  -\Gamma_k U_{1,k} \Sigma_k \Lambda_k^\top\\ & \quad - \Gamma_k (C_k P^{\star x}_{k|k}-R_k U_{2,k} M_{2,k}^\top G_{2,k-1}^\top)) &= 0 \\ 
\frac{\partial{\mathcal{L}}}{\partial \Lambda_k} &= -2 \overline{L}_k \Gamma_k U_{1,k} &= 0
\end{align*}
Solving the above linear system of equations and simplifying, we obtain the optimal gain matrix \eqref{eq:Lk}. 
\end{proof}

One choice of $\Gamma_k$ (first proposed in \cite{Darouach.1997} using the singular value decomposition of $\hat{R}^{-\frac{1}{2}}_k C_k G_{2,k-1}=\tilde{U}_k \tilde{\Sigma}_k \tilde{V}_k^\top $) such that $\Gamma_k \tilde{R}^\star_k \Gamma_k^\top$ has full rank, is given by
\begin{align}
\Gamma_k=\begin{bmatrix} 0 & I_{p_R} \end{bmatrix} \tilde{U}_k^\top \hat{R}_k^{-\frac{1}{2}},
\end{align}
where $\hat{R}_k$ and $\tilde{R}^\star_k$ are defined in the text following \eqref{eq:cov}, and $p_R=l-p-p_{H_{k-1}}$.
With this choice of $\Gamma_k$, we obtain $\Gamma_k \tilde{R}^\star_k \Gamma_k^\top=I_{p_R}$ which is invertible. Following the procedure in \cite[Appendix]{Darouach.1997}, it can be shown that \eqref{eq:Lk} reduces to
\begin{align}
L_k=K_k(I_l-H_{1,k} M_{1,k}^\star)^\top \hat{R}_k^{-1}.
\end{align}
with $M^\star_{1,k}:=\Sigma_k^{-1} (U_{1,k}^\top \hat{R}_k^{-1} N_k U_{1,k})^{-1} U_{1,k}^\top \hat{R}_k N_k$, which is independent of $\tilde{U}_k$ and as such, the ``expensive" singular value decomposition step can be bypassed. Another choice would be to use the Moore-Penrose pseudoinverse ($^\dagger$) such that $\check{R}_k=(\tilde{R}^\star_k)^\dagger$. Equivalently, we have $L_k=L_k \begin{bmatrix} U_{1,k} & U_{2,k} \end{bmatrix} \begin{bmatrix} U_{1,k}^\top \\ U_{2,k}^\top \end{bmatrix}=\tilde{L}_k U_{2,k}^\top$ where  we defined 
\begin{align}
\tilde{L}_k:=L_{k} U_{2,k}=K_k(I_l-H_{1,k} M_{1,k}^\star)^\top \hat{R}_k^{-1} U_{2,k}.
\end{align}

In addition, we can compute the (cross-)covariances as 
\begin{subequations}  \label{eq:Pxd1}
\begin{empheq}[left=\empheqlbrace]{align}
\nonumber  P^{xd {\rm \, I}}_{1,k} &= (P^{dx {\rm \, I}}_{1,k})^\top = -P^x_{k|k} C_{1,k}^\top M_{1,k}^\top +L_{k} R_k T_{1,k}^\top M_{1,k}^\top \hspace{0.3cm} \hspace{-1.35cm}\\& \qquad \qquad \ \ \ = -P^x_{k|k} C_{1,k}^\top M_{1,k}^\top \\
\nonumber P^{xd {\rm \, II}}_{1,k} &= (P^{dx {\rm \, II}}_{1,k})^\top = L_{k} R_k T_{1,k}^\top M_{1,k}^\top\\
\nonumber &  \qquad- L_k R_k T_{2,k}^\top M_{2,k}^\top G_{2,k-1}^\top C_{1,k}^\top M_{1,k}^\top\\& \qquad -(I-L_k C_k)P^{\star x}_{k|k} C_{1,k}^\top M_{1,k}^\top
\end{empheq}
\end{subequations} \vspace{-0.7cm}
\begin{align}
P^{xd}_{2,k-1} &= (P^{dx}_{2,k-1})^\top=-P^x_{k-1|k-1} A_{k-1}^\top C_{2,k}^\top M_{2,k}^\top \\
\nonumber & \qquad \qquad \qquad \quad -P^{xd}_{1,k-1} G_{1,k-1}^\top C_{2,k}^\top M_{2,k}^\top \\ 
P^d_{12,k-1}&=(P^d_{21,k-1})^\top=-P^{dx}_{1,k-1} A_{k-1}^\top C_{2,k}^\top M_{2,k}^\top\\
\nonumber &  \qquad \qquad \qquad \quad -P^{d}_{1,k-1} G_{1,k-1}^\top C_{2,k}^\top M_{2,k}^\top\\  
P^d_{k} :&=\begin{bmatrix} V_{1,k} & V_{2,k} \end{bmatrix} \begin{bmatrix} P^d_{1,k} & P^d_{12,k} \\ P^d_{21,k} & P^d_{2,k} \end{bmatrix} \begin{bmatrix} V_{1,k}^{\, \top} \\ V_{2,k}^{\, \top} \end{bmatrix}
\end{align}
 where we can apply $L_k=L_k U_{2,k} U_{2,k}^\top$ from Lemma \ref{lem:unbiased} and \eqref{eq:R12} such that $L_k R_k T_{1,k}^\top M_{1,k}^\top=0$ which resulted in the simplification of \eqref{eq:Pxd1} shown above.
% where we can apply $L_k=L_k U_{2,k} U_{2,k}^\top$ from Lemma \ref{lem:unbiased} to further simplify \eqref{eq:Pxd1} such that $L_k R_k T_{1,k}^\top M_{1,k}^\top=0$.

\subsection{Global optimality of ULISE} \label{sec:optimality}

In the following, we relax the recursivity assumption of ULISE for both the state and input estimates and consider $\hat{x}_{k|k}$ and $\hat{d}_{k}$ to be the most general linear combination of the unbiased initial state estimate $\hat{x}_{0|0}$ and $\mathcal{Z}_k$ given in \eqref{eq:matrices}. We first prove that the state update of ULISE has the same optimal form as the filter proposed in \cite[Remark 3]{Cheng.2009}, through which the claim of global optimality of the state estimate over the class of all linear estimators follows from \cite{Kerwin.Prince.2000}. Then, we prove that the input estimate is also globally optimal, which completes the proof of Theorem \ref{thm:globalULISE}.

%We have shown in Section \ref{sec:inputEst} that the input estimates are unbiased. Thus, the unknown input can be viewed as consisting of a known component given by the input estimate, and a zero-mean random variable with known variance which can be dealt in the same manner as with process and measurement noises. Hence, the global optimality of a filter for a system with unknown inputs is closely related to the optimality of its state and error covariance update laws. With this in mind, we proceed to prove the global optimality of ULISE by first proving the equivalence of the state and covariance update of ULISE with the filter proposed in \cite{Cheng.2009}, through which the claim of global optimality over the class of all linear estimators follows.

\begin{proof}[Proof of Theorem \ref{thm:globalULISE}]To this end, we rearrange the latter form of \eqref{eq:stateEst} of state estimation for ULISE with unknown inputs estimated with \eqref{eq:variant1} and \eqref{eq:d2}, to obtain
\begin{align}
\nonumber \hat{x}_{k|k}=&\hat{A}_{k-1} \hat{x}_{k-1|k-1}+ B_{k-1} u_{k-1} +G_{1,k-1} M_{1,k-1} z_{1,k-1} \\
\nonumber &+\overline{K}_k (z_{2,k}-C_{2,k}(\hat{A}_{k-1}\hat{x}_{k-1|k-1}+B_{k-1} u_{k-1} \\ &+G_{1,k-1} M_{1,k-1} z_{1,k-1}))\\
\overline{K}_k=& G_{2,k-1} M_{2,k} + \tilde{L}_k (I-C_{2,k} G_{2,k-1} M_{2,k})
\end{align}
where $\hat{A}_{k-1}=A_{k-1}-G_{1,k-1} M_{1,k-1} C_{1,k-1}$, as previously defined. Repeating the procedure in Section \ref{sec:measUpdate}, $\tilde{L}_k=(\tilde{P}_k C_{2,k}^\top- G_{2,k-1} M_{2,k} \tilde{R}_{2,k}) \overline{N}_k^\top (\overline{N}_k \tilde{R}_{2,k} \overline{N}_k^\top)^{-1}\overline{\Gamma}_k$ and
\begin{align}
P^x_{k|k}&=(I-\overline{K}_k C_{2,k}) \tilde{P}_{k} (I-\overline{K}_k C_{2,k})^\top +\overline{K}_k R_{2,k} \overline{K}_k^\top \hspace{-0.2cm}
\end{align}
where  $\overline{N}_k:= \overline{\Gamma}_k(I-C_{2,k}G_{2,k-1} M_{2,k})$, $\tilde{R}_{2,k}:=C_{2,k} \tilde{P}_k C_{2,k}^\top +R_{2,k}$ and $ \overline{\Gamma}_k$ is an arbitrary matrix such that $\overline{N}_k \tilde{R}_{2,k} \overline{N}_k^\top$ has full rank. 
Thus, the ULISE's state and state covariance update is almost identical to the one considered in \cite{Cheng.2009},  in which only state estimation is considered. The only difference is in the choice of $M_{2,k}$, where $M_{2,k}$ is replaced by $\tilde{M}_{2,k}:=(C_{2,k} G_{2,k-1})^\dagger$ in \cite{Cheng.2009}. More importantly, the state update law is of the optimal form \cite[Remark 3]{Cheng.2009} from which the global optimality of the state estimate over the linear class of estimators according to \cite{Kerwin.Prince.2000}.  

%Moreover, the generalized least squares input estimates are also the best linear unbiased estimates (see Theorem \ref{thm:mse}).
To show that the input estimate is also globally optimal, we consider the input estimate $\hat{d}^g_{k-1}$ to be the most general linear combination of the unbiased initial state estimate $\hat{x}_{0|0}$, as well as $\mathcal{Z}_{1,k}$ and $\mathcal{Z}_{2,k}$ given in \eqref{eq:matrices}. Since $\tilde{z}_{1,i}$ and $\tilde{z}_{2,i}$ as defined for \eqref{eq:ztilde1} and \eqref{eq:ztilde2} are linear combinations of $\hat{x}_{0|0}$, $\mathcal{Z}_{1,i}$ and $\mathcal{Z}_{2,i}$, and of $\hat{x}_{0|0}$, $\mathcal{Z}_{1,i-1}$ and $\mathcal{Z}_{2,i}$, respectively, $\hat{d}^g_{k-1}$ can be expressed as
\vspace{-0.275cm}\begin{align} 
\hat{d}^g_{k-1}= \chi_{0}(k) \hat{x}_{0|0} + \sum_{i=1}^k \chi_{1,i}(k) \tilde{z}_{1,i} + \sum_{i=1}^k \chi_{2,i}(k) \tilde{z}_{2,i}. \label{eq:dhat_g} 
\end{align} 
Clearly, if $\chi_{1,k-1}(k)=V_{1,k-1} M_{1,k-1}$ and $\chi_{2,k}(k)=V_{2,k-1} M_{2,k}$ where $M_{1,k-1}$ and $M_{2,k}$ are as in \eqref{eq:Mk} and \eqref{eq:Mk2}, and if $\chi_{0}(k)$, $\chi_{1,k}(k)$, $\{\chi_{1,i}(k)\}_{i=0}^{k-2}$ and $\{\chi_{2,i}(k)\}_{i=0}^{k-1}$ are zero, then $\hat{d}^g_{k-1}$ is unbiased. To show the converse, we suppose that $\hat{d}^g_{k-1}$ is unbiased, i.e.,  $\mathbb{E}[\hat{d}^g_{k-1}]=V_{1,k-1} d_{1,k-1}+V_{2,k-1} d_{2,k-1}$. Since $d_k$ can take on any arbitrary value and $z_{1,k}$ is a function of $d_{1,k}$, $\chi_{1,k}(k)=0$ such that $\hat{d}^g_{k-1}$ remains unbiased. Moreover, the first measurements containing $d_{1,k-1}$ and $d_{2,k-1}$ are $z_{1,k-1}$ and $z_{2,k}$, then $\mathbb{E}[\chi_{1,k-1}(k)\tilde{z}_{1,k-1}]=V_{1,k-1}d_{1,k-1}$ and $\mathbb{E}[\chi_{2,k}(k)\tilde{z}_{2,k}]=V_{2,k-1}d_{2,k-1}$. Consequently, $\chi_{1,k-1}(k)=V_{1,k-1} M_{1,k-1}$ and $\chi_{2,k}(k)=V_{2,k-1} M_{2,k}$. %where $M_{1,k-1}$. 
Moreover, for $\hat{d}^g_{k-1}$ to be unbiased, $\chi_{0}(k)=0$, $\{\chi_{1,i}(k)\}_{i=0}^{k-2}=0$ and $\{\chi_{2,i}(k)C_{2,i} G_{2,i-1}\}_{i=0}^{k-1}=0$ must hold. Finally, we prove that the mean squared error $\mathbb{E}[\|d_{k-1}-d^g_{k-1}\|_2^2]$ is minimized when $\{\chi_{2,i}(k)\}_{i=0}^{k-1}=0$. From the unbiasedness conditions of $\hat{d}^g_{k-1}$ and from \eqref{eq:dhat_g}, we have $d_{k-1}-d_{k-1}^g=\tilde{d}_{k-1}-\sum^{k-1}_{i=0} \chi_{2,i}(k) \tilde{z}_{2,i}$ where $\tilde{d}_{k}$ is as defined above Lemma \ref{lem:unbiased}. Since it is straightforward to verify (as in \cite[Lemmas 1 and 2]{Kerwin.Prince.2000}) that $\mathbb{E}[\tilde{d}_{k}(\chi_{2,i}(k) \tilde{z}_{2,i})^\top]=0$ for all $i\leq k$, it follows that
\small \vspace{-0.65cm}
\begin{align*}
&\mathbb{E}[\|d_{k-1}-d^g_{k-1}\|_2^2]\\
&={\rm tr}\{\mathbb{E}[(\tilde{d}_{k-1}-\sum^{k-1}_{i=0} \chi_{2,i}(k) \tilde{z}_{2,i})(\tilde{d}_{k-1}-\sum^{k-1}_{i=0} \chi_{2,i}(k) \tilde{z}_{2,i})^\top]\}\\
&={\rm tr}\{\mathbb{E}[\tilde{d}_{k-1}\tilde{d}_{k-1}^\top]\}+\mathbb{E}[\|\sum^{k-1}_{i=0} \chi_{2,i}(k) \tilde{z}_{2,i}\|_2^2]\
\end{align*}\normalsize \vspace{-0.15cm}
%$\hat{x}_{i|i}$ is a linear combination of $\hat{x}_{0|0}$, $\mathcal{Z}_{1,i}$ and $\mathcal{Z}_{2,i}$, whereas $\hat{x}_{i|i-1}$ is a linear combination of $\hat{x}_{0|0}$, $\mathcal{Z}_{1,i}$ and $\mathcal{Z}_{2,i}$
where the first term is minimized by ULISE as is shown in \eqref{eq:traced} and Theorem \ref{thm:mse}, while the latter term is minimized when $\sum^{k-1}_{i=0} \chi_{2,i}(k) \tilde{z}_{2,i}=0$, which occurs when $\{\chi_{2,i}(k)\}^{k-1}_{i=0}=0$, as desired.
%Thus, the ULISE variant's state and state covariance update is identical to the one considered in \cite{Cheng.2009}, in which only state estimation is considered, and thus, since we assume that ${\rm rk}(C_{2,k} G_{2,k-1})=p-p_{H_{k-1}}$, which satisfies the assumption in \cite{Cheng.2009} that $${\rm rk}\begin{bmatrix} G_{2,k-1} \\ C_{2,k} G_{2,k-1}\end{bmatrix}={\rm rk}(C_{2,k} G_{2,k-1}),$$ the global optimality claim in \cite{Cheng.2009} applies. Moreover, since the input estimates satisfy the Gauss-Markov assumption (\cite{kailath.2000}), the input estimates are also the best linear unbiased estimates. 
 Thus, Theorem \ref{thm:globalULISE} holds.
 \end{proof}

\begin{rem}
We also conclude that the state estimator in \cite{Cheng.2009} implicitly estimates the unknown input, i.e., with \eqref{eq:variant1} and \eqref{eq:d2}, although the replacement of $M_{2,k}$ by $\tilde{M}_{2,k}$ in \eqref{eq:d2} is tantamount to using an ordinary least squares (OLS) estimate instead of the generalized least squares (GLS) estimate, resulting in the same expected estimate but the estimate does not have minimum variance (see discussion in \cite[pp. 223-224]{Draper.Smith.1998}). Furthermore, ULISE provides a family of optimal state estimators parameterized by $\Gamma_k$, whereas the filter in \cite{Cheng.2009} provides a specific solution by choosing $\overline{N}_k$ as the left null matrix of $C_{2,k} G_{2,k}$, i.e., $\overline{N}_k={\rm Null} ((C_{2,k} G_{2,k})^\top)^\top$. More importantly, we have shown that the decorrelation constraint assumed in \cite{Cheng.2009}, such that only $z_{2,k}$ can be used in the state update to avoid obtaining a suboptimal estimator, is justified as a direct consequence of the unbiasedness constraint in Lemma \ref{lem:unbiased}, i.e., $L_k U_{1,k}=0$. By extension, ULISE is also less restrictive than the filter in \cite{Darouach.2003}.
%We also conclude that the state estimator in \cite{Cheng.2009} implicitly estimates the unknown input in the BLUE sense, i.e., with \eqref{eq:variant1} and \eqref{eq:d2}. Furthermore, ULISE provides a family of optimal state estimators parametrized by $\Gamma$, whereas the filter in \cite{Cheng.2009} provides a specific solution by choosing $\overline{N}_k$ as the left null matrix of $C_{2,k} G_{2,k}$, i.e., $\overline{N}_k={\rm Null} ((C_{2,k} G_{2,k})^\top)^\top$. More importantly, we have shown that the decorrelation constraint assumed in \cite{Cheng.2009}, such that only $z_{2,k}$ can be used in the state update to avoid obtaining a suboptimal estimator, is justified as a direct consequence of the unbiasedness constraint in Lemma \ref{lem:unbiased}, i.e., $L_k U_{1,k}=0$. By extension, ULISE is also less restrictive than the filter in \cite{Darouach.2003}.
In addition, the unknown input estimates are BLUE, thus, ULISE is globally optimal over the class of all linear unbiased \emph{state and input} estimates for systems with unknown inputs.
However, the same cannot be said of PLISE, as can be seen in the examples of Section \ref{sec:examples}.
\end{rem}

\subsection{Stability of ULISE} \label{sec:stableULISE}

In this section, we prove the stability of the ULISE filter by first reducing
the linear time-varying system with unknown inputs to an equivalent system
without unknown inputs. Then, we use existing results on the stability of the Kalman filter \cite[Section 5]{Anderson.Moore.1981} to obtain the sufficient conditions for the stability of the original system.

\begin{proof}[Proof of Theorem \ref{thm:stableULISE}]
We begin by reducing the system with unknown inputs to one without unknown inputs. From \eqref{eq:stateEst} and \eqref{eq:sysY}, we obtain $\tilde{x}_{k|k}=\tilde{x}^\star_{k|k}-\tilde{L}_k (C_{2,k} \tilde{x}^\star_{k|k}+v_{2,k})$. Then, substituting \eqref{eq:dtilde} into \eqref{eq:x_kk-1} and the above equation, and rearranging, we obtain
\begin{align}
\nonumber \tilde{x}_{k|k}&=\overline{A}_{k-1} \tilde{x}_{k-1|k-1} + \overline{w}_{k-1}-\tilde{L}_k(C_{2,k} \overline{A}_{k-1} \tilde{x}_{k-1|k-1} \\ \label{eq:xtilde_stab}
& \quad +C_{2,k} \overline{w}_k + v_{2,k}), 
\end{align}
where $\overline{A}_{k-1}=(I-G_{2,k-1} {M}_{2,k}C_{2,k}) \hat{A}_{k-1}$ and $\overline{w}_k=-(I-G_{2,k-1} {M}_{2,k} C_{2,k}) (G_{1,k-1} M_{1,k-1} v_{1,k-1}-w_{k-1} -G_{2,k-1} {M}_{2,k} v_{2,k} $. As it turns out, the state estimate error dynamics above is the same for a Kalman filter \cite{KalmanF.1960} for a linear system without unknown inputs:
%\begin{align*}
$x^e_{k+1}=\overline{A}_k x^e_k +\overline{w}_k; \ y^e_k=C_{2,k} x^e_k+v_{2,k}.$
%\end{align*}
Since the objective for both systems is the same, i.e., to obtain
an unbiased minimum-variance filter, they are equivalent
systems from the perspective of optimal filtering. However, the noise terms of this equivalent system are correlated, i.e., $\mathbb{E}[\overline{w}_k v_{2,k}^\top]=-G_{2,k-1} {M}_{2,k} R_{2,k}$. To transform the system further into one without correlated noise, we employ a common trick of adding a zero term since $y^e_k-C_{2,k} x^e_k-v_{2,k}=0$ to obtain
\begin{align*}
\nonumber x^e_{k+1}&=\overline{A}_k x^e_k +\overline{w}_k - G_{2,k-1} {M}_{2,k} (y^e_k-C_{2,k} x^e_k - v_{2,k})\\
&=\bar{\bar{A}}_k x^e_k +\bar{\bar{u}}_k +\bar{\bar{w}}_k\\
y^e_k&=C_{2,k} x^e_k+v_{2,k}
\end{align*}
where $\bar{\bar{A}}_k=\overline{A}_k+G_{2,k-1} {M}_{2,k} C_{2,k}$, $\bar{\bar{u}}_k= - G_{2,k-1} {M}_{2,k} y^e_k$ is a known input and $\bar{\bar{w}}_k=\overline{w}_k+ G_{2,k-1} {M}_{2,k} v_{2,k}$. The noise terms $\bar{\bar{w}}_k$ and $v_{2,k}$ are uncorrelated with covariances $\bar{\bar{Q}}_k:=\mathbb{E}[\bar{\bar{w}}_k \bar{\bar{w}}^{\top}_k]=(I-G_{2,k-1} {M}_{2,k} C_{2,k})\hat{Q}_{k-1} (I-G_{2,k-1} {M}_{2,k} C_{2,k})^\top$, $R_{2,k}$ and $\mathbb{E}[\bar{\bar{w}}_k v_{2,k}^\top]=0$, where $M_{2,k}$ and $\hat{Q}_{k-1}$ are as defined in Theorem \ref{thm:mse}.
%This new equivalent system has uncorrelated noise terms $\bar{\bar{w}}_k$ and $v_{2,k}$ with covariances $\bar{\bar{Q}}_k:=\mathbb{E}[\bar{\bar{w}}_k \bar{\bar{w}}^{\top}_k]=(I-G_{2,k-1} {M}_{2,k} C_{2,k})\hat{Q}_{k-1} (I-G_{2,k-1} {M}_{2,k} C_{2,k})^\top$, $R_{2,k}$ and $\mathbb{E}[\bar{\bar{w}}_k v_{2,k}^\top]=0$, where $M_{2,k}$ and $\hat{Q}_{k-1}$ are as defined in Theorem \ref{thm:mse}.

Ideally, if we can compute $\bar{\bar{A}}$ and $\bar{\bar{Q}}$ prior to applying the ULISE algorithm, then the uniform detectability and stabilizability conditions of \cite[Section 5]{Anderson.Moore.1981} can be directly applied to obtain the desired stability property. However, this is not the case as these matrices depend on $P^x_{k-1|k-1}$ which is not available a priori. Thus, we substitute $M_{2,k}$ in \eqref{eq:d2} with $\tilde{M}_{2,k}:=(C_{2,k} G_{2,k-1})^\dagger$ to obtain $\tilde{A}_k:=(I-G_{2,k-1} \tilde{M}_{2,k}C_{2,k}) \hat{A}_{k-1}+G_{2,k-1} \tilde{M}_{2,k} C_{2,k}$ and $\tilde{Q}_k:=(I-G_{2,k-1} \tilde{M}_{2,k} C_{2,k})\hat{Q}_{k-1} (I-G_{2,k-1} \tilde{M}_{2,k} C_{2,k})^\top$. This removes the dependence on $P^x_{k-1|k-1}$ from the uniform detectability and stabilizability tests in Theorem \ref{thm:stableULISE}. %Note that this substitution results in the state estimator given in \cite{Cheng.2009} with $\overline{N}_k={\rm Null} ((C_{2,k} G_{2,k})^\top)^\top$, which as shown in Section \ref{sec:optimality} is a suboptimal version of the ULISE algorithm, where \cite{Cheng.2009} used the ordinary least squares (OLS) estimate of $\hat{d}_{2,k-1}$ in place of the generalized least squares (GLS) estimate employed in ULISE (i.e., with $M_{2,k}$ given in Theorem \ref{thm:mse}).

From \cite[Lemma 5.1 \& Corollary 5.2]{Anderson.Moore.1981}, if $(\tilde{A}_k,C_{2,k})$ is uniformly detectable, then the corresponding filter error covariance $P^{x,sub}_{k|k}$ is bounded. By the optimality of the ULISE algorithm, it follows that the ULISE error covariance $P^{x}_{k|k}$ and $\tilde{L}_k$ are bounded. Next, by \cite[Theorems 4.3 \& 5.3]{Anderson.Moore.1981}, the uniform stability of $(\tilde{A}_k,\tilde{Q}_k^{\frac{1}{2}})$ and the boundedness of $\tilde{L}_k$ implies that the filter (with $\tilde{L}_k$ but with $\tilde{M}_{2,k}$ in the input estimate) is exponentially stable. Finally, using the fact that the ordinary and generalized least squares input estimates have the same expected value (see, e.g., \cite[pp. 223-224]{Draper.Smith.1998}), it can be verified from \eqref{eq:xtilde_stab} that $\mathbb{E}[\tilde{x}_{k|k}]=(I-\tilde{L}_k C_{2,k} )\overline{A}_{k-1} \mathbb{E}[\tilde{x}_{k-1|k-1}]=(I-\tilde{L}_k C_{2,k} )\tilde{A}_{k-1} \mathbb{E}[\tilde{x}_{k-1|k-1}]
$,
%\begin{align*}
%\mathbb{E}[\tilde{x}_{k|k}]&=(I-\tilde{L}_k C_{2,k} )\overline{A}_{k-1} \mathbb{E}[\tilde{x}_{k-1|k-1}]\\
%&=(I-\tilde{L}_k C_{2,k} )\tilde{A}_{k-1} \mathbb{E}[\tilde{x}_{k-1|k-1}],
%\end{align*}
from which it follows that the uniform stability of $(\tilde{A}_k,\tilde{Q}_k^{\frac{1}{2}})$ and the boundedness of $\tilde{L}_k$ also implies that ULISE is exponentially stable.
\end{proof}
%\vspace{-0.3cm}
%Thus, since the ULISE algorithm and the above linear system without unknown inputs are equivalent in the context of optimal filtering, the sufficient conditions for stability of the Kalman filter \cite[Section 5]{Anderson.Moore.1981} can be directly to yield the sufficient conditions given in Theorem \ref{thm:stableULISE}.
%
%substituting $M_{2,k}$ in \eqref{eq:d2} with $\tilde{M}_{2,k}:=(C_{2,k} G_{2,k-1})^\dagger$. %(G_{2,k-1}^\top C_{2,k}^\top C_{2,k} G_{2,k-1})^{-1}G_{2,k-1}^\top C_{2,k}^\top $. 
%This substitution serves to remove the dependence on $P^x_{k-1|k-1}$ (not available a priori) from the uniform detectability and stabilizability tests in Theorem \ref{thm:stableULISE}. Note that this substitution results in the state estimator given in \cite{Cheng.2009} with $\overline{N}_k={\rm Null} ((C_{2,k} G_{2,k})^\top)^\top$, which as shown in Section \ref{sec:optimality} is a specific solution of the ULISE algorithm. % and an input estimator that is suboptimal (does not satisfy Gauss-Markov theorem \cite[pp. 96-98]{kailath.2000}).
%\end{proof}}

Next, we consider the time-invariant case, for which uniform detectability and uniform stabilizability reduce to standard definitions of detectability and stabilizability \cite{Peters.Iglesias.1999}. Thus, the sufficient conditions of Theorem \ref{thm:convULISE} follow directly. In addition, necessary and sufficient conditions can be obtained for the time-invariant case. Noting the similarity of ULISE to the state estimator in \cite{Cheng.2009} and the conditions given in \cite{Darouach.1997} is independent of the choice of $M_{2,k}$ or  $\tilde{M}_{2,k}$, it can be shown that the convergence and stability conditions are as given in Theorem \ref{thm:convULISE}.

\subsection{Stability of PLISE} \label{sec:stablePLISE}

Unfortunately, the `more complex' structure of PLISE renders the proof approach in the previous section for the stability of ULISE for the time-varying case not applicable. Instead of taking this problem head-on, we choose to only consider the stability of the PLISE variant for the case of linear time invariant systems in Theorem \ref{thm:convPLISE}, which will proven next. %Given the equivalence of ULISE to the state estimator in \cite{Cheng.2009}, we can also directly use the convergence and stability condition given in \cite{Darouach.1997}. Hence, Theorem \ref{thm:convULISE} holds.

\begin{proof}[Proof of Theorem \ref{thm:convPLISE}]
To proof the sufficiency of the conditions in Theorem \ref{thm:convPLISE} for the PLISE variant of the unified filter, we consider a suboptimal version of PLISE that utilized a non-BLUE $\hat{d}_2$ by assuming that $\tilde{R}_2=I$, and thus, $M_2$ becomes $\tilde{M}_2$ (similar to the assumption of \cite{Cheng.2009}). Then, we rewrite \eqref{eq:Pstar2} to obtain the associated algebraic Riccati equation as 

\vspace{-0.5cm}
\small
\begin{align*}
\hat{P}^{\star x}=(F^s-K^s C_2) \hat{P}^{\star x} (F^s-K^s C_2)^\top +K^s \Theta K^{s \top} +Q^s
\end{align*}
\normalsize
where $K^s=\hat{N} A L U_2-\hat{S} \Theta^{-1}$, while $F^s$, $Q^s$, $\hat{N}$, $\hat{S}$ and $\Theta$ are as defined in Theorem \ref{thm:convPLISE}. Using the results in \cite{Chan.1984, deSouza.1986}, the error covariance matrix exponentially converges to a unique stabilizing solution of the algebraic Riccati equation if and only if $(F^s,C_2)$ is detectable and $(F^s,Q^{s \frac{1}{2}})$ has no unreachable modes on the unit circle (Condition \eqref{cond2}). To obtain Condition \eqref{cond1} from the detectability of $(F^s,C_2)$, we use the following identities:

\vspace{-0.5cm}
\small
\begin{align*}
{\rm rk} \begin{bmatrix} zI-F^s \\ C_2 \end{bmatrix}&={\rm rk} \begin{bmatrix} I & -s \Theta^{-1} \\ 0 & I \end{bmatrix} \begin{bmatrix} zI-F^s \\ C_2 \end{bmatrix}\\
&=\begin{bmatrix} zI - \hat{N} \hat{A} \\ C_2 \end{bmatrix} =n, \, \forall z \in \mathbb{C}, |z| \geq 1 \\
{\rm rk} \begin{bmatrix} zI-\hat{A} & -G_2 \\ C_2 & 0\end{bmatrix}&={\rm rk} \begin{bmatrix} zI-\hat{A} & -G_2 \\ C_2 & 0\end{bmatrix} \begin{bmatrix} I & 0 \\ - \hat{N} C_2 \hat{A} & I\end{bmatrix}\\
&={\rm rk} \begin{bmatrix} zI-\hat{N}\hat{A} & -G_2 \\ C_2 & 0\end{bmatrix}\\
&={\rm rk} \begin{bmatrix} zI-\hat{N}\hat{A}  \\ C_2 \end{bmatrix}+p-p_H\\
&=n+p-p_H, \, \forall z \in \mathbb{C}, |z| \geq 1,
\end{align*}
\normalsize
the latter of which is equivalent to strong detectability of the system by Theorem \ref{thm:det}. Since the suboptimal version of PLISE admits a bounded steady-state solution, the error covariance, and hence the estimate errors of PLISE remain bounded because by the optimality of PLISE, $P^{x} \leq (I-L C)\hat{P}^{\star x}(I-LC)^\top +LRL^\top +(I-LC)G_2 \tilde{M} U_2^\top R L^\top +L R U_2 \tilde{M}^\top G_2^\top (I-L C)^\top $ where $L=(P^{\star x} C^\top - G_{2} \tilde{M}_{2} U_{2}^\top R)\check{R} (I_l-H_{1} M_{1}^\star)$, $H_{1}=U_{1} \Sigma$, $M^\star_{1}=\Sigma^{-1} (U_{1}^\top \check{R} U_{1})^{-1} U_{1}^\top \check{R}$, $\check{R}:=\Gamma^\top (\Gamma \tilde{R}^\star \Gamma^\top)^{-1} \Gamma$, $\tilde{R}^\star=C P^{x \star} C^\top +R-C G_2 \tilde{M}_2 U_2^\top R-R U_2 \tilde{M}_2^\top G_2^\top C^\top$ and $\Gamma$ is such that $\Gamma \tilde{R}^\star \Gamma^\top$ has full rank.
\end{proof}

\section{Connection to existing literature} \label{sec:connection}
In this section, we show that ULISE and PLISE reduce to estimators that are closely related to the estimators in existing literature in the following special cases.
\subsection{Special Case 1: $H_k$ has full rank} 
In this special case, ${\rm rk}(H_k)=p$ and the singular value decomposition of $H_k=\begin{bmatrix} U_{1,k} & U_{2,k} \end{bmatrix} \begin{bmatrix} \Sigma_k \\ 0 \end{bmatrix} V_{1,k}^\top=U_{1,k} \Sigma_k V_{1,k}^\top$. Thus, $V_{2,k}$ is an empty matrix and correspondingly $G_{2,k}$, $d_{2,k}$, $M_{2,k}$ and $P^d_{2,k}$ are also empty matrices. From \eqref{eq:xstar}, \eqref{eq:Pstar} and \eqref{eq:cov}, we have $\hat{x}^\star_{k|k}=\hat{x}_{k|k-1}$,
\begin{align}
P^x_{k|k}&=(I-L_k C_k) P^x_{k|k-1} (I-L_k C_k)^\top +L_k R_k L_k^\top \hspace{-0.1cm} \\
\tilde{R}_k^\star&=\tilde{R}_k:= C_k P^x_{k|k-1} C_k^\top +R_k\\
\nonumber P^{\star x}_{k|k}&=P^x_{k|k-1}:=\mathbb{E}[(x_k-x_{k|k-1})(x_k-x_{k|k-1})^\top]\\
&= A_{k-1}  P^x_{k-1|k-1} A_{k-1}^\top + G_{k-1} P^d_{k} G_{k-1}^\top\\
\nonumber &+A_{k-1} P^{xd}_{k-1} G_{k-1}^\top +G_{k-1} P^{xd \top}_{k-1} A_{k-1}^\top +Q_{k-1}\\
\nonumber P^d_k&=V_{1,k} P^d_{1,k} V_{1,k}^\top \\
&= (H_k^\top U_{1,k} (T_{1,k} \tilde{R}_k T_{1,k}^\top)^{-1} U_{1,k}^\top H_k)^{-1} \label{eq:PdH1}
\end{align} \vspace{-1.5cm}
\small \begin{subequations} 
\begin{empheq}[left=\empheqlbrace]{align}  
P^{xd {\rm \, I}}_{k} &= P^{xd {\rm \, I}}_{1,k} V_{1,k}^\top = L_{k} R_{k} M_{k}^\top- P^x_{k|k} C_{k}^\top M_{k}^\top\\
P^{xd {\rm \, II}}_{k} &= P^{xd {\rm \, II}}_{1,k} V_{1,k}^\top = L_k \tilde{R}_k M_{k}^\top-P^{x}_{k|k-1} C_{k}^\top M_{k}^\top  \label{eq:PdH2} 
\end{empheq}
\end{subequations}
\normalsize
where  we have defined
\begin{align}
\nonumber &M_k:=V_{1,k} M_{1,k} T_{1,k}\\
\nonumber &=(H_k^\top U_{1,k} (T_{1,k} \tilde{R}_k T_{1,k}^\top)^{-1} U_{1,k}^\top H_k^\top)^{-1} \\ &\qquad H_k^\top U_{1,k} (T_{1,k} \tilde{R}_k T_{1,k}^\top)^{-1} T_{1,k}. \label{eq:MkH1}
\end{align}
Since $\tilde{R}_k$ has full rank, $\Gamma_k$ can be chosen as the identity matrix and the state update and input estimates are
\begin{align}
\hspace{-0.1cm} \hat{x}_{k|k-1}&=A_{k-1} \hat{x}_{k-1|k-1}+B_{k-1} u_{k-1}+G_{k-1} \hat{d}_{k-1}\\
\hat{x}_{k|k}&=\hat{x}_{k|k-1}+L_k(y_k-C_k \hat{x}_{k|k-1}-D_k u_k)
\end{align}
\begin{subequations}
\begin{empheq}[left=\empheqlbrace]{align}
\hat{d}_k^{\rm \, I} &= M_k (y_k-C_k \hat{x}_{k|k}-D_k u_k)\\
\hat{d}_k^{\rm \, II} &= M_k (y_k-C_k \hat{x}_{k|k-1} -D_k u_k)
\end{empheq}
\end{subequations}
with $L_k=P^x_{k|k-1} C_k^\top \tilde{R}_k^{-1} (I-H_k (H_k^\top \tilde{R}_k^{-1} H_k)^{-1} H_k^\top \tilde{R}_k^{-1})$.

Comparing the above equations with the filters in \cite{Gillijns.2007b,Yong.Zhu.Frazzoli.2013}, we note that ULISE variant is closely related to the filter proposed in \cite{Yong.Zhu.Frazzoli.2013}, with the main difference in \eqref{eq:PdH1} and \eqref{eq:MkH1}, which would be equivalent if $T_{1,k}=U_{1,k}^\top$ and $U_{1,k} (T_{1,k} \tilde{R}_k T_{1,k}^\top)^{-1} U_{1,k}^\top=\tilde{R}_k$, which is only true when $U_{2,k}$ is an empty matrix, i.e., when $H_k$ has full row rank.

On the other hand, the PLISE variant is closely related to the filter in \cite{Gillijns.2007b}. Similarly, the only differences lie in \eqref{eq:PdH1}, \eqref{eq:PdH2}  and \eqref{eq:MkH1}, and the filters are equivalent when $H_k$ has full row rank, which also leads to $L_k \tilde{R}_k M_k^\top=0$.

\subsection{Special Case 2: $H_k=0$} 
\vspace{-0.2cm}
In this case, no transformation of the output equations and no decomposition of the unknown input vector is necessary. The $U_{1,k}$ and $V_{1,k}$ are empty matrices while $U_{2,k}$ and $V_{2,k}$ are identity. Thus, ULISE and PLISE reduce to the same state and covariance update equations given by
\begin{align}
\hat{x}_{k|k-1}&=A_{k-1} \hat{x}_{k-1|k-1} +B_{k-1} u_{k-1}\\
\hat{x}^\star_{k|k}&=\hat{x}_{k|k-1}+G_{k-1} \hat{d}_{k-1}\\
\hat{x}_{k|k}&=\hat{x}^\star_{k|k}+L_k(y_k-C_k \hat{x}^\star_{k|k}-D_k u_k)\\
\hat{d}_{k-1}&=M_k (y_k-C_k \hat{x}_{k|k-1}-D_k u_k) \\ 
P^x_{k|k-1}&=A_{k-1} P^x_{k-1|k-1} A_{k-1}^\top +Q_{k-1}\\
\nonumber P^{\star x}_{k|k}&= (I-G_{k-1}M_k C_k)P^x_{k|k-1} (I-G_{k-1}M_kC_k)^\top\\
& \quad +G_{k-1}M_k R_k M_k^\top G_{k-1}^\top\\
P^d_k&=(G_{k-1}^\top C_k^\top \tilde{R}_k^{-1} C_k G_{k-1})^{-1}\\
P^{xd}_k&=-P^x_{k-1|k-1} A_{k-1}^\top C_k^\top M_k^\top\\
P^x_{k|k}&=P^{\star x}_{k|k}+ L_k \tilde{R}^\star_k L_k-L_k S_k^\top -S_k L_k^\top
\end{align}
where $\tilde{R}_k=C_k P^x_{k|k-1} C_k^\top +R_k$, $\tilde{R}_k^\star=C_k P^{\star x}_{k|k} C_k^\top +R_k -C_k G_{k-1} M_k R_k-R_k M_k^\top G_{k-1}^\top C_k^\top$, $S_k=-G_{k-1} M_k R_k +P^{\star x}_{k|k} C_k^\top$, $M_k=(G_{k-1}^\top C_k^\top \tilde{R}_k^{-1} C_k G_{k-1})^{-1} G_{k-1}^\top C_k^\top \tilde{R}_k^{-1}$ and $L_k=(P^{\star x}_{k|k} C_k^\top-G_{k-1} M_k R_k) \check{R}_k$. The above equations are identical to the filter derived in \cite{Gillijns.2007} for systems without direct feedthrough, therefore, ULISE and PLISE are generalizations of the filter in \cite{Gillijns.2007} to systems with direct feedthrough, and by extension, of the filters in \cite{Kitanidis.1987,Darouach.1997}.

\vspace{-0.2cm}
\subsection{Special Case 3: $G_k=0$ and $H_k=0$} \label{sec:kalman} 
\vspace{-0.2cm}
When $G_k=0$ and $H_k=0$, the filter gain $L_k$ reduces to the \emph{Kalman filter} gain $L_k=P^x_{k|k-1} C_k^\top \tilde{R}_k^{-1}$ where $\tilde{R}_k=C_k P_{k|k-1} C_k^\top +R_k$, while the state and covariance update reduces to the Kalman filter equations:
\begin{align}
& \hspace{-0.2cm}  \hat{x}_{k|k}=A_{k-1}\hat{x}_{k-1|k-1}+L_k(y_k-C_k A_{k-1} \hat{x}_{k-1|k-1})\\
& \hspace{-0.2cm} P^x_{k|k-1}=A_{k-1} P^x_{k-1|k-1} A_{k-1}^\top + Q_{k-1}\\
& \hspace{-0.2cm} P^x_{k|k}=(I-L_k C_k) P^x_{k|k-1} (I-L_k C_k)^\top +L_k R_k L_k^\top
\end{align}

\section{Illustrative Examples} \label{sec:examples}
\subsection{Fault Identification}
In this example, we consider the state estimation and fault identification problem when the system dynamics is plagued by faults, $d_k$, that can either influence the system dynamics through the input matrix $G_k$ or the outputs through the feedthrough matrix $H_k$, as well as zero-mean Gaussian white noise. Thus, the objective is to estimate the states of the system for the sake of continued operation in spite of the faults, and to identify the faults that the system is experiencing for self-repair or maintenance purposes. Specifically, the linear discrete-time problems we consider are based on the system given in \cite{Cheng.2009}, which is similar to the failure detection problem first considered in \cite{Keller.1996}, with six different $H$ matrices to illustrate the effect of parameter changes on filter performance:

\small \vspace{-0.5cm}
\begin{align*}
A &= \begin{bmatrix} 0.5 & 2 & 0 & 0 & 0\\ 0 & 0.2 & 1 & 0 &1 \\ 0 & 0 & 0.3 & 0 & 1 \\ 0 & 0 & 0 & 0.7 & 1 \\ 0 & 0 & 0 & 0 & 0.1\end{bmatrix}; \begin{array}{ll}
 B &= 0_{5 \times 1}; \\
C &= I_5;\\
D &= 0_{5 \times 1}; \end{array} \;
%\\ %\end{align*} \begin{align*} % \\
 G = \begin{bmatrix} 1 & 0 & -0.3 \\ 1 & 0 & 0 \\ 0& 0 & 0 \\0& 0 & 0 \\0& 0 & 0 \end{bmatrix};\\ \quad
Q &= 10^{-4} \begin{bmatrix} 1 & 0 & 0 & 0 & 0  \\ 0 & 1 & 0.5& 0 & 0\\ 0& 0.5 & 1 & 0 & 0 \\ 0 & 0 & 0 & 1 & 0 \\ 0 & 0 & 0 & 0 & 1 \end{bmatrix}; \, %\\ %\end{align*} \begin{align*} %\\
R = 10^{-2}\begin{bmatrix}  1 & 0 & 0 & 0.5 & 0\\0 & 1 & 0 & 0 & 0.3\\0 & 0 & 1 & 0 & 0\\ 0.5 & 0 & 0 & 1 & 0 \\ 0 & 0.3 & 0 & 0 & 1\end{bmatrix} \hspace{-0.1cm};\\
%\end{align*} \normalsize
%with six different $H$ matrices \small \vspace{-2cm}
%\begin{align*}
H^1&=\begin{bmatrix} 0 & 0 & 1 \\ 0 & 0 & 0 \\ 0 & 1 & 0 \\ 0 & 0 & 0 \\ 0 & 0 & 0 \end{bmatrix}; \quad
H^2=\begin{bmatrix} 0 & 0 & 1 \\ 0 & 0 & 0 \\ 0 & 1 & 0 \\ 0 & 0 & 0 \\ 1 & 0 & 0 \end{bmatrix}; \quad
H^3=\begin{bmatrix} 0 & 0 & 0 \\ 0 & 0 & 0 \\ 0 & 1 & 0 \\ 0 & 0 & 0 \\ 1 & 0 & 0 \end{bmatrix}, \\ %\end{align*} \begin{align*} %\\
H^4&=\begin{bmatrix} 0 & 0 & 0 \\ 1 & 0 & 0 \\ 0 & 1 & 0 \\ 0 & 0 & 0 \\ 0 & 0 & 0 \end{bmatrix}; \quad
H^5=\begin{bmatrix} 0 & 0 & 0 \\ 0 & 0 & 0 \\ 0 & 1 & 0 \\ 0 & 0 & 1 \\ 0 & 0 & 0 \end{bmatrix}; \quad
H^6=\begin{bmatrix} 0 & 0 & 0 \\ 1 & 0 & 0 \\ 0 & 1 & 0 \\ 0 & 0 & 1 \\ 0 & 0 & 0 \end{bmatrix}.
\end{align*}
\normalsize
%\begin{align*}
%A &= \begin{bmatrix} 0.5 & 2 & 0 & 0 & 0\\ 0 & 0.2 & 1 & 0 &1 \\ 0 & 0 & 0.3 & 0 & 1 \\ 0 & 0 & 0 & 0.7 & 1 \\ 0 & 0 & 0 & 0 & 0.1\end{bmatrix};\quad \begin{array}{ll}
% B &= 0_{5 \times 1}; \\
%C &= I_5;\\
%D &= 0_{5 \times 1}; \end{array}
% \\
% G &= \begin{bmatrix} 1 & 0 & -0.3 \\ 1 & 0 & 0 \\ 0& 0 & 0 \\0& 0 & 0 \\0& 0 & 0 \end{bmatrix}; \quad
%Q = 10^{-4} \begin{bmatrix} 1 & 0 & 0 & 0 & 0  \\ 0 & 1 & 0.5& 0 & 0\\ 0& 0.5 & 1 & 0 & 0 \\ 0 & 0 & 0 & 1 & 0 \\ 0 & 0 & 0 & 0 & 1 \end{bmatrix};\\
%R &= 10^{-2}\begin{bmatrix}  1 & 0 & 0 & 0.5 & 0\\0 & 1 & 0 & 0 & 0.3\\0 & 0 & 1 & 0 & 0\\ 0.5 & 0 & 0 & 1 & 0 \\ 0 & 0.3 & 0 & 0 & 1\end{bmatrix}
%\end{align*}
%with six different $H$ matrices
%\begin{align*}
%H^1=\begin{bmatrix} 0 & 0 & 1 \\ 0 & 0 & 0 \\ 0 & 1 & 0 \\ 0 & 0 & 0 \\ 0 & 0 & 0 \end{bmatrix}; \quad
%H^2=\begin{bmatrix} 0 & 0 & 1 \\ 0 & 0 & 0 \\ 0 & 1 & 0 \\ 0 & 0 & 0 \\ 1 & 0 & 0 \end{bmatrix}; \quad
%H^3=\begin{bmatrix} 0 & 0 & 0 \\ 0 & 0 & 0 \\ 0 & 1 & 0 \\ 0 & 0 & 0 \\ 1 & 0 & 0 \end{bmatrix}, \\
%H^4=\begin{bmatrix} 0 & 0 & 0 \\ 1 & 0 & 0 \\ 0 & 1 & 0 \\ 0 & 0 & 0 \\ 0 & 0 & 0 \end{bmatrix}; \quad
%H^5=\begin{bmatrix} 0 & 0 & 0 \\ 0 & 0 & 0 \\ 0 & 1 & 0 \\ 0 & 0 & 1 \\ 0 & 0 & 0 \end{bmatrix}; \quad
%H^6=\begin{bmatrix} 0 & 0 & 0 \\ 1 & 0 & 0 \\ 0 & 1 & 0 \\ 0 & 0 & 1 \\ 0 & 0 & 0 \end{bmatrix},
%\end{align*}
%to illustrate the effect of parameter changes to the performance of the estimators. 
With the above $H$ matrices, the invariant zeros of the matrix pencil $\begin{bmatrix} zI-\hat{A} & -G_2\\  C_2 & 0 \end{bmatrix}$ are respectively $\{0.3,0.8\}$, $\{0.1,0.3,0.5,0.7,0.8\}$, $\emptyset$, $\{0.3,-0.8\}$, $\emptyset$ and $\{0.1,0.7, 0.3,-0.8,0.35\}$. Thus, all six systems are strongly detectable. Moreover, the direct feedthrough matrices of the second and sixth systems, $H^2$ and $H^6$, have full rank.

The unknown inputs used in this example are 
\vspace{-0.5cm}
\small
\begin{align*}
d_{k,1}&=\left\{
\begin{array}{c l}
    1, & \quad 500 \leq k \leq 700\\
    0, & \quad {\rm otherwise}
\end{array}\right.\\
d_{k,2}&=\left\{
\begin{array}{c l}
    \frac{1}{700}(k-100), & \quad 100 \leq k \leq 800\\
    0, & \quad {\rm otherwise}
\end{array}\right.\\
d_{k,3}&=\left\{
\begin{array}{c l}
   3, &  500\leq k \leq 549, 600 \leq k \leq 649, 700 \leq k \leq 749\\
   -3, &  550\leq k \leq 599, 650 \leq k \leq 699, 750 \leq k \leq 799\\
    0, &  {\rm otherwise}.
\end{array}\right.
\end{align*}
\normalsize

To illustrate the performance of the unified simultaneous input and state estimators, measured by the steady-state trace of the error covariance matrices, we compare the performance of the following filters:  (i) Cheng et al. filter \cite{Cheng.2009}, augmented by estimates the unknown input in the BLUE sense, i.e., with \eqref{eq:variant1} and \eqref{eq:d2} (CYWZ), (ii) ULISE from Section \ref{sec:filter}, and (iii) PLISE from Section \ref{sec:filter}, as well as the filters for systems with full-rank $H$ matrix: (iv) Gillijns and De Moor filter (GDM) \cite{Gillijns.2007b}, (iv) Fang et al. filter (FSY) \cite{Fang.2011} and (v) Yong et al. filter (YZF) \cite{Yong.Zhu.Frazzoli.2013}. The simulations were implemented in MATLAB on a 2.2 GHz Intel Core i7 CPU.

\begin{figure}[!t]
\begin{center}
\includegraphics[scale=0.3575,trim=30mm 10mm 20mm 10mm]{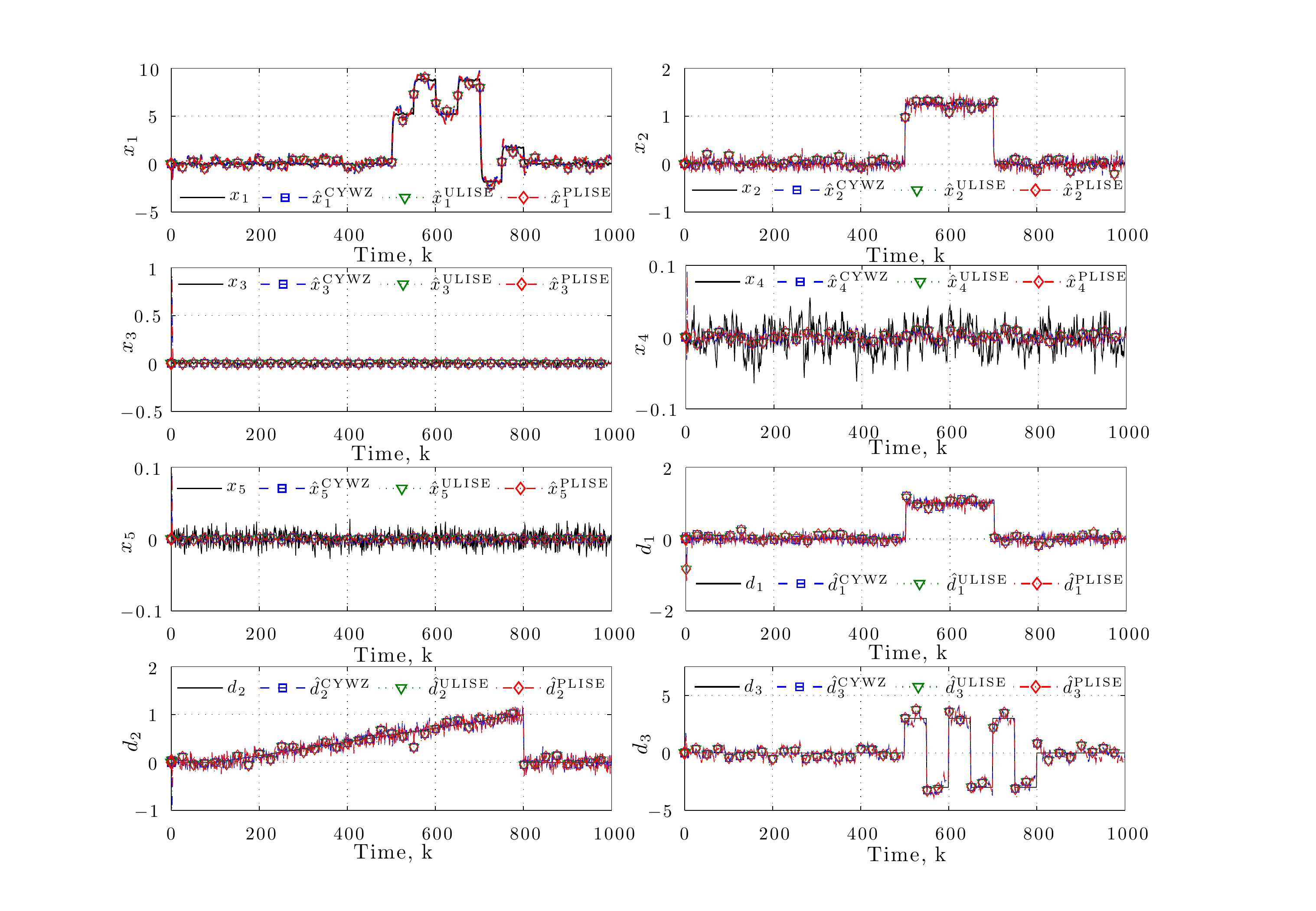}
\caption{Actual states $x_1$, $x_2$, $x_3$, $x_4$, $x_5$ and its estimates, as well as unknown inputs $d_1$, $d_2$ and $d_3$ and its estimates.\label{fig:inputs} }
\end{center}
\end{figure}

\begin{figure}[!t]
\begin{center}
\includegraphics[scale=0.37,trim=25mm 55mm 20mm 0mm]{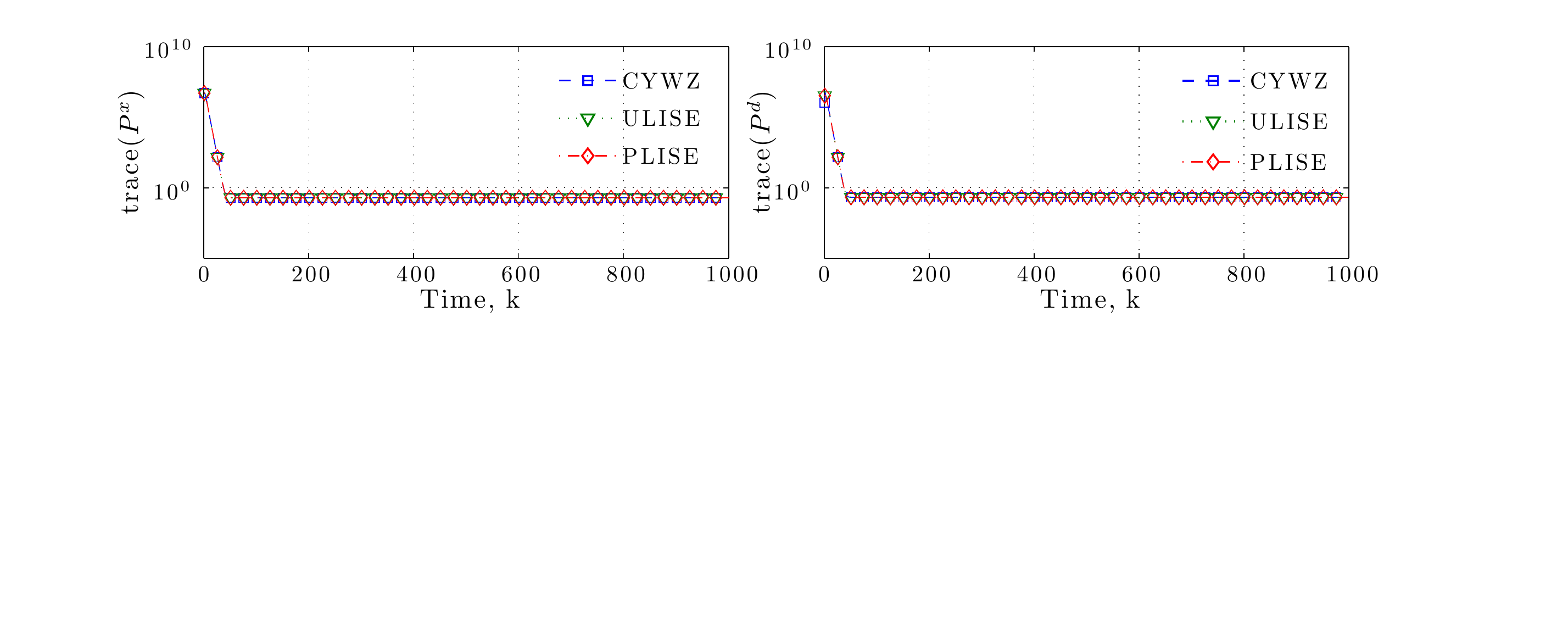}
\caption{Trace of estimate error covariance of states, tr($P^x$), and unknown inputs, tr($P^d$). \label{fig:variances}}% for the first 500 time steps.\label{fig:variances}}
\end{center}
%\vspace{-0.1cm}
\end{figure}

Figure \ref{fig:inputs} shows a comparison of the input and state estimation of the first three MVU estimators for the first system with $H^1$. In this case, these estimators were successful at estimating the states as well as the unknown inputs. It does appear from Figure \ref{fig:variances} all three estimators produces the same steady-state error covariances. However, if we consider the results of all six systems in Table \ref{tab:compare}, we observe that PLISE is  outperformed by CYWZ and ULISE.  Note also that ULISE are consistently the best filters, which agrees with the claim in Section  \ref{sec:optimality} of being globally optimal over the class of all linear unbiased state and input estimates for systems with unknown inputs, while CYWZ performs just as well, which shows that in this particular example, the replacement of the generalized least squares estimate of $d_{2,k}$ with the ordinary least squares estimate have little impact on the filter performance.

\begin{table}[t] \tabcolsep=0.07cm \tiny
\caption{Steady-state Performance of CYWZ, ULISE, PLISE, GDM, FSY and YZF. \label{tab:compare} \vskip0.1cm}
\centering
\begin{tabular}{| c | l || c | c | c | c | c || c | c| c|} \hline
\multicolumn{2}{|c||}{} &  {$P^x_{11}$} & {$P^x_{22}$} & {$P^x_{33}$} & $P^x_{44}$ & $P^x_{55}$ & $P^d_{11}$ & $P^d_{22}$ & $P^d_{33}$ \\
\hline
\multirow{6}{*}{$H^1$} & CYWZ & 0.1843  &  0.0091  &  0.0002 &   0.0004  &  0.0001 &   0.0099 &   0.0102  &  0.1923  \\
&  ULISE  & 0.1843 &   0.0091 &   0.0002 &   0.0004 &   0.0001  &  0.0099 &   0.0102  &  0.1923   \\
&  PLISE  & 0.1843  &  0.0091  &  0.0002 &   0.0004 &   0.0001  &  0.0099  &  0.0102 &   0.1923   \\
&  GDM  & N/A   & N/A & N/A & N/A   & N/A & N/A & N/A   & N/A  \\
&  FSY  & N/A   & N/A & N/A & N/A   & N/A & N/A & N/A   & N/A  \\
&  YZF  & N/A   & N/A & N/A & N/A   & N/A & N/A & N/A   & N/A   \\ \hline
\multirow{6}{*}{$H^2$} & CYWZ & 0.1494  &  0.0052 &   0.0002  &  0.0004  &  0.0001 &   0.0097 &   0.0102 &   0.1574  \\
&  ULISE  & 0.1494  &  0.0052 &   0.0002  &  0.0004  &  0.0001 &   0.0097 &   0.0102 &   0.1574    \\
&  PLISE  & 0.1614 & 0.0053  &  0.0002 &   0.0004 &   0.0001 &   0.0102  &  0.0102  &  0.1889   \\
&  GDM  & 0.1494  &  0.0052 &   0.0002  &  0.0004 &   0.0001  &  0.0097 &   0.0102 &   0.1574   \\
&  FSY  & 0.1724 &   0.0108 &   0.0002  &  0.0004 &   0.0001 &   0.0097 &   0.0102  &  0.1648  \\
&  YZF  & 0.1494  &  0.0052 &   0.0002  &  0.0004 &   0.0001  &  0.0097 &   0.0102 &   0.1574   \\ \hline
\multirow{6}{*}{$H^3$} & CYWZ & 0.0076  &  0.0052 &   0.0002 &   0.0004  &  0.0001 &   0.0097 &   0.0102  &  0.3906 \\
&  ULISE  & 0.0076  &  0.0052 &   0.0002 &   0.0004  &  0.0001 &   0.0097 &   0.0102  &  0.3906 \\
&  PLISE  & 0.0076  &  0.0053  &  0.0002  &  0.0004  &  0.0001 &   0.0102 &   0.0102  &  0.3961   \\
&  GDM  & N/A   & N/A & N/A & N/A   & N/A & N/A & N/A   & N/A  \\
&  FSY  & N/A   & N/A & N/A & N/A   & N/A & N/A & N/A   & N/A  \\
&  YZF  & N/A   & N/A & N/A & N/A   & N/A & N/A & N/A   & N/A  \\  \hline
\multirow{6}{*}{$H^4$} & CYWZ & 0.0076 &   0.0257  &  0.0002 &   0.0004  &  0.0001  &  0.0348 &   0.0102 &   0.4925 \\
&  ULISE  & 0.0076 &   0.0257  &  0.0002 &   0.0004  &  0.0001  &  0.0348 &   0.0102 &   0.4925 \\
&  PLISE  & 0.0076 &   0.0258  &  0.0002  &  0.0004 &   0.0001 &   0.0349  &  0.0102  &  0.4925   \\
&  GDM  & N/A   & N/A & N/A & N/A   & N/A & N/A & N/A   & N/A  \\
&  FSY  & N/A   & N/A & N/A & N/A   & N/A & N/A & N/A   & N/A  \\
&  YZF  & N/A   & N/A & N/A & N/A   & N/A & N/A & N/A   & N/A  \\  \hline
\multirow{6}{*}{$H^5$} & CYWZ & 0.0079 &   0.0074  &  0.0002  &  0.0004  &  0.0001 &   0.0089  &  0.0102  &  0.0099 \\
&  ULISE  & 0.0079 &   0.0074  &  0.0002  &  0.0004  &  0.0001 &   0.0089  &  0.0102  &  0.0099 \\
&  PLISE  & 0.0079  &  0.0074 &   0.0002 &   0.0004  &  0.0001  &  0.0089 &   0.0102  &  0.0150   \\
&  GDM  & N/A   & N/A & N/A & N/A   & N/A & N/A & N/A   & N/A  \\
&  FSY  & N/A   & N/A & N/A & N/A   & N/A & N/A & N/A   & N/A  \\
&  YZF  & N/A   & N/A & N/A & N/A   & N/A & N/A & N/A   & N/A  \\  \hline
\multirow{6}{*}{$H^6$} & CYWZ & 0.0076 &   0.0218 &   0.0002 &   0.0004  &  0.0001 &   0.0309  &  0.0102 &   0.0097 \\
&  ULISE  & 0.0076 &   0.0218 &   0.0002 &   0.0004  &  0.0001 &   0.0309  &  0.0102 &   0.0097 \\
&  PLISE  & 0.0078 &   0.0257  &  0.0002  &  0.0004  &  0.0001  &  0.0368   & 0.0102  &  0.0165   \\
&  GDM  & 0.0076  &  0.0218 &   0.0002 &   0.0004 &   0.0001 &   0.0309 &   0.0102 &   0.0097  \\
&  FSY  & 0.0315 &   0.0232 &   0.0002  &  0.0004  &  0.0001 &   0.0310  &   0.0102  &  0.0100  \\
&  YZF  & 0.0076 &   0.0218 &   0.0002  &  0.0004  &  0.0001  &  0.0309  &  0.0102 &   0.0097  \\  \hline
\end{tabular}
\end{table}

On the other hand, when the direct feedthrough matrix has full rank, as with $H^2$ and $H^6$, GDM and YZF performed just as well as CYWZ and ULISE, which is consistent with the claim of global optimality of GDM in \cite{Hsieh.2010}. In both examples, the intentionally suboptimal FSY filter performs better than PLISE at estimating the unknown inputs, but is worse than PLISE when estimating the system states.

\subsection{Multi-vehicle Tracking}
In this second example, we consider the problem of the position and velocity tracking of multiple vehicles, for e.g., at an intersection, with partial information about the decisions of the vehicles as well as faulty sensor readings. This can be particularly useful for the design of intelligent transportation systems. To simplify the problem, we consider the scenario with two vehicles, in which each vehicle only has access to its own control input, thus, the input of the other vehicle is unknown. Furthermore, the velocity measurement of the vehicle is corrupted by a time-varying bias, which is also unknown. Thus, we model the linear continuous-time model of the coupled system as:
\scriptsize
\begin{align*}
\begin{bmatrix}
\dot{p} \\ \ddot{p} \\ \dot{q} \\ \ddot{q}
\end{bmatrix} &= \begin{bmatrix} 0 & 1 & 0 & 0 \\ 0 & -0.1 & 0 & 0 \\ 0 & 0 & 0 & 1 \\ 0 & 0 & 0 & -0.1 \end{bmatrix}\begin{bmatrix}
p \\ \dot{p} \\ q \\ \dot{q}
\end{bmatrix}+\begin{bmatrix} 0 \\ 0 \\ 0 \\ 1 \end{bmatrix} u+\begin{bmatrix} 0 & 0 \\ 1 & 0 \\ 0 & 0 \\ 0 & 0 \end{bmatrix} \begin{bmatrix} d_1 \\ d_2 \end{bmatrix} + \begin{bmatrix} 0 \\ w_1 \\ 0 \\ w_2 \end{bmatrix} \\
y&=\begin{bmatrix} 1 & 0 & 0 & 0 \\ 0 & 1 & 0 & -1 \\ 0 & 0 & 1 & 0 \\ 0 & 0 & 0 & 1  \end{bmatrix} \begin{bmatrix}
p \\ \dot{p} \\ q \\ \dot{q}
\end{bmatrix}+ \begin{bmatrix} 0 & 0 \\ 0 & 0 \\ 0 & 0 \\ 0 & 1 \end{bmatrix} \begin{bmatrix} d_1 \\ d_2 \end{bmatrix} + v
\end{align*}
\normalsize
where $p$ and $\dot{p}$, and $q$ and $\dot{q}$, are the displacements and velocities of the uncontrolled and controlled vehicle, respectively. $d_1$ is the unknown input of the uncontrolled  vehicle while $d_2$ represents the unknown time-varying bias. The intensities of the zero mean, white Gaussian noises, $w=\begin{bmatrix} 0 & w_1 & 0 & w_2\end{bmatrix}^\top$ and $v$, are given by:
\begin{align*}
Q_c &= 10^{-4} \begin{bmatrix} 0 & 0 & 0 & 0  \\ 0 & 1.6 &  0 & 0\\ 0& 0 & 0 & 0 \\ 0 & 0 & 0 & 0.9 \end{bmatrix}; \
R_c = 10^{-4} \begin{bmatrix} 1 & 0 & 0 & 0  \\ 0 & 0.16 &  0 & 0\\ 0& 0 & 0.9 & 0 \\ 0 & 0 & 0 & 2.5 \end{bmatrix}.
\end{align*}
Since the proposed filter is for discrete systems, we first convert the continuous dynamics to a discrete equivalent model with sample time $\triangle t= 0.01s$, assuming zero-order hold for the known and unknown inputs, $u$ and $d$:
\begin{align*}
x_{k+1}&=A_d x_k + B_d u_k + G_d d_k + w_{d,k}\\
y_k&=C_d x_k + H_d d_k + v_{d,k}
\end{align*}
where $x=\begin{bmatrix} p & \dot{p} & q & \dot{q} \end{bmatrix}^\top$, $k=0,1,2,\hdots$ and $t=k \triangle t$, while the system matrices as well as noise covariances can be computed, e.g., using conversion algorithms involving matrix exponentials as in \cite{DeCarlo.1989,VanLoan.1978}, to obtain:
\begin{align*}
A_d &= \begin{bmatrix} 1 & 0.01 & 0 & 0 \\ 0 & -0.999 & 0 & 0 \\ 0 & 0 & 1 & 0.01 \\ 0 & 0 & 0 & 0.999 \end{bmatrix};\
 B_d = \begin{bmatrix} 0 \\ 0 \\ 0 \\ 0.01 \end{bmatrix}; \\
 C_d &= \begin{bmatrix} 1 & 0 & 0 & 0 \\ 0 & 1 & 0 & -1 \\ 0 & 0 & 1 & 0 \\ 0 & 0 & 0 & 1 \end{bmatrix}; \
 G_d = \begin{bmatrix} 0 & 0 \\ 0.01 & 0 \\ 0 & 0 \\ 0 & 0 \end{bmatrix}; \quad
 H_d = \begin{bmatrix} 0 & 0 \\ 0 & 0 \\ 0 & 0 \\ 0 & 1 \end{bmatrix};\\% \end{align*} \begin{align*} 
Q_d &= 10^{-5} \begin{bmatrix} 0.0000 &   0.0008   &      0   &      0 \\
    0.0008 &   0.1598   &      0    &     0\\
         0   &      0  &  0.0000  &  0.0004\\
         0   &      0  &  0.0004  &  0.0899\end{bmatrix}; \;
R_d = R_c,
\end{align*}
with $d_{1,k}$ and $d_{2,k}$ as shown in Figure \ref{fig:inputs2} (where $t=k \Delta t$).

From Figure \ref{fig:inputs2}, we observe that both variants of the filter proposed in this paper successfully estimate the system states and the unknown inputs, which consist of the input of the uncontrolled vehicle and the time-varying measurement bias. The slight difference between the two variants can be seen in Figure \ref{fig:variances2} where the rate of convergence of trace of the unknown input estimate error covariance of the PLISE variant is slightly slower.

\begin{figure}[!t]
\begin{center}
\includegraphics[scale=0.37,trim=25mm 7mm 40mm 10mm,clip]{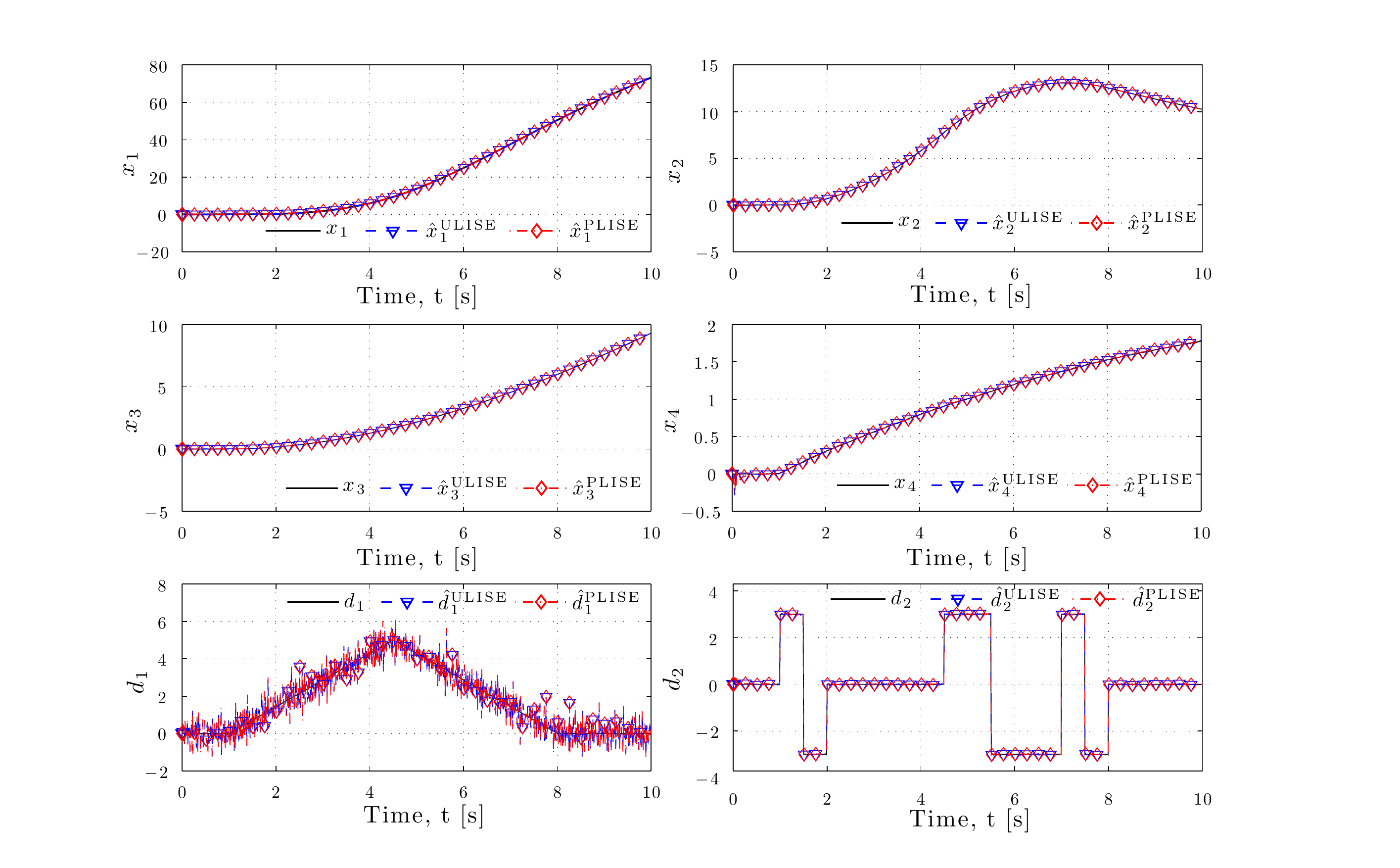}
\caption{Actual states $x_1$, $x_2$, $x_3$, $x_4$ and its estimates, as well as unknown inputs $d_1$, $d_2$, and its estimates.\label{fig:inputs2} }
\end{center}
\end{figure}

\begin{figure}[!t]
\begin{center}
\includegraphics[scale=0.37,trim=30mm 60mm 40mm 5mm]{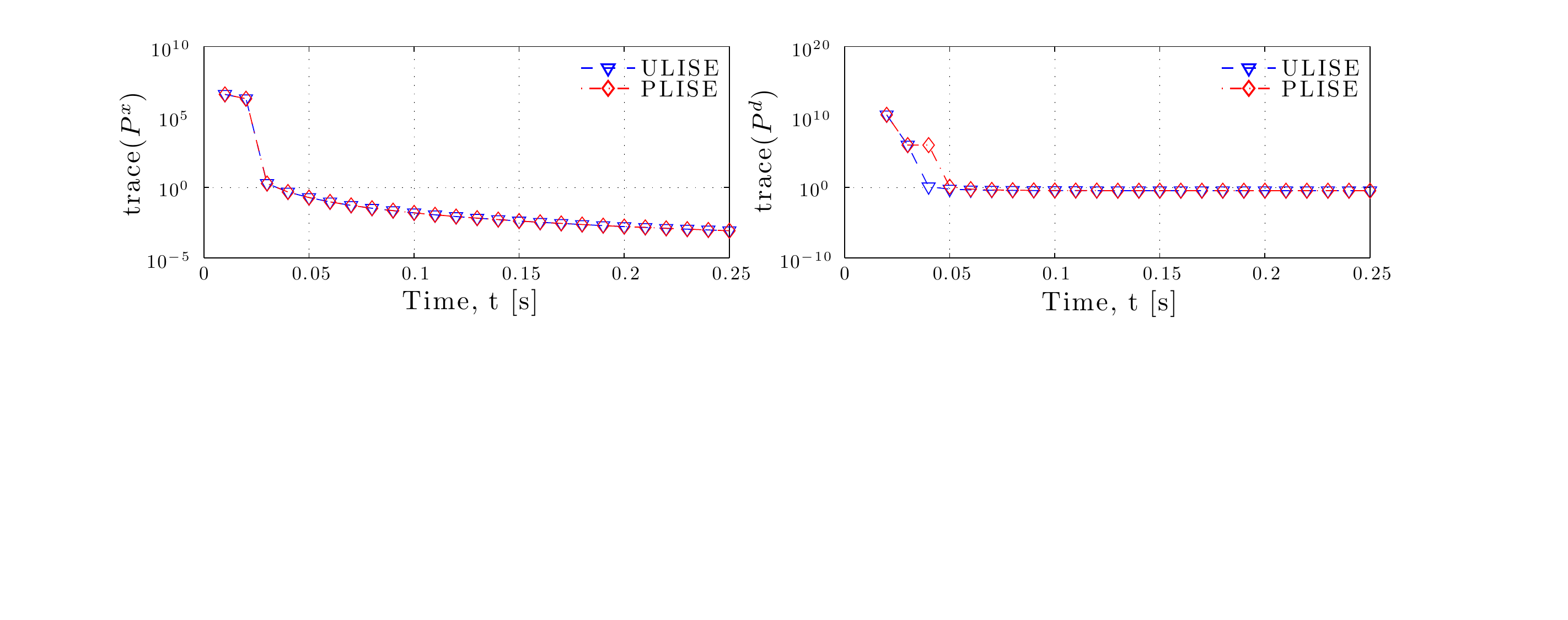}
\caption{Trace of estimate error covariance of states, tr($P^x$), and unknown inputs, tr($P^d$) for the first $0.25s$.\label{fig:variances2}}
\end{center}
\end{figure}

\vspace{-0.2cm}
\section{Conclusion} \label{sec:conclusion}
\vspace{-0.1cm}
This paper presented a unified filter for simultaneously estimating the states and unknown inputs in an unbiased minimum-variance sense for linear discrete-time stochastic systems, without any restriction on the direct feedthrough matrix of the system. Two variants of the filter is proposed, one of which uses the propagated state estimate for unknown input estimation (PLISE), and the other with the updated state estimate (ULISE). %By means of the connection of the ULISE variant to a globally optimal state estimator, 
We proved that ULISE is also globally optimal over the class of all linear unbiased state and input estimates for systems with unknown inputs and provided stability conditions for the filter, which are shown to be closely related to the strong detectability of the system. %Convergence and stability conditions for the time-invariant case of both variants are also derived, which is shown to be closely related to the strong detectability of the considered system. 
Simulation results have shown that ULISE was the best estimator in all the test trials, whereas PLISE, though is not globally optimal, performed reasonably well.

A possible future direction is the extension of the current unified filter to linear continuous-time systems, switched systems and nonlinear systems.

\vspace{-0.2cm}
\section*{Acknowledgments}
\vspace{-0.25cm}
The work presented in this paper was supported in part by the National Science Foundation, grant \#1239182.

\vspace{-0.2cm}
\bibliographystyle{unsrt}

\bibliography{biblio}

\end{document}